\documentclass{article}
\usepackage{latexsym}
\usepackage{amsmath,amsthm,amsfonts,amsbsy,amsgen,amscd,amsxtra,amstext,amsopn,amssymb}
\usepackage[all]{xy}

\theoremstyle{plain}
\newtheorem{lemma}{Lemma}[section]	
\newtheorem{theorem}[lemma]{Theorem}
\newtheorem{proposition}[lemma]{Proposition}
\newtheorem{corollary}[lemma]{Corollary}

\theoremstyle{definition}
\newtheorem{definition}[lemma]{Definition}

\theoremstyle{remark}
\newtheorem{remark}[lemma]{Remark}
\newtheorem{example}[lemma]{Example}

\def\N{\mathbb{N}}
\def\Z{\mathbb{Z}}
\def\R{\mathbb{R}}

\def\C{\mathbb{C}}

\def\proj{\mathbb{P}}

\renewcommand{\a}{\alpha}

\renewcommand{\b}{\beta}
\renewcommand{\d}{\delta}
\newcommand{\D}{\Delta}

\newcommand{\G}{\Gamma}

\renewcommand{\L}{\Lambda}

\newcommand{\var}{\varphi}

\newcommand{\Om}{\Omega}
\newcommand{\om}{\omega}

\def\mZ{{\mathcal Z}}
\def\mI{{\mathcal I}}
\def\mH{{\mathscr H}}

\renewcommand{\hat}{\widehat}

\renewcommand{\tilde}{\widetilde}

\def\Oh{{\cal O}}

\def\algorithm{\begin{center}
               \begin{minipage}{6in}
               \begin{tabbing}
               \marks}
\def\falgorithm{\end{tabbing}
                \end{minipage}
                \end{center}}
\def\marks{nn\= nn\= nn\= nn\= nn\= nn\= nn\= \kill}

\def\PSPACE{{\sf PSPACE}}

\def\PR{{\rm P}_{\kern-1pt\R}}

\def\PC{{\rm P}_{\kern-1pt\C}}
\def\NPR{{\rm NP}_{\kern-1pt\R}}
\def\NPC{{\rm NP}_{\kern-2pt\C}}
\def\DNPR{{\rm DNP}_{\kern-1pt\R}}
\def\DNPC{{\rm DNP}_{\kern-2pt\C}}
\def\PAR{{\rm PAR}_{\kern-1pt\R}}
\def\PHR{{\rm PH}_{\kern-1pt\R}}
\def\DPHR{{\rm DPH}_{\kern-1pt\R}}

\def\FPR{{\rm FP}_{\kern-1pt\R}}
\def\FPC{{\rm FP}_{\kern-1pt\C}}
\def\FPAR{{\rm FPAR}_{\kern-0.4pt\R}}
\def\FPARC{{\rm FPAR}_{\kern-0.4pt\C}}

\def\CPRi{{\rm \#P}_{\kern-2pt\R}}
\def\CPCi{{\rm \#P}_{\kern-2pt\C}}

\def\FEASR{{\mbox{\sc Feas}_{\kern-0.5pt\R}}}  
\def\FEASRbit{{\mbox{\sc Feas}^0_{\kern-1pt\R}}}
\def\SAS{{\mbox{\sc SAS}_{\kern-0.5pt\R}}}
\def\SASbit{{\mbox{\sc SAS}_{\kern-1pt\R}^0}}
\def\HNC{{\mbox{\sc HN}_{\kern-1pt\C}}}
\def\HNCbit{{\mbox{\sc HN}^0_{\kern-1pt\C}}}
\def\QASC{{\mbox{\sc QAS}_{\kern-1pt\C}}}

\def\DIMR{{\mbox{\sc Dim}_{\kern-0.5pt\R}}}
\def\DIMC{{\mbox{\sc Dim}_{\kern-0.5pt\C}}}
\def\DIMadd{{\mbox{\sc Dim}_{\kern-0.5pt\add}}}
\def\DIMRbit{{\mbox{\sc Dim}^0_{\kern-0.5pt\R}}}
\def\DIMCbit{{\mbox{\sc Dim}^0_{\kern-0.5pt\C}}}
\def\REACH{{\mbox{\sc Reach}_{\kern-0.5pt\R}}}
\def\REACHbit{{\mbox{\sc Reach}^0_{\kern-0.5pt\R}}}
\def\CREACHbit{{\mbox{\sc CReach}^0_{\kern-0.5pt\R}}}

\def\REACH{{\mbox{\sc Reach}_{\kern-0.5pt\R}}}
\def\REACHbit{{\mbox{\sc Reach}^0_{\kern-0.5pt\R}}}
\def\CREACHbit{{\mbox{\sc CReach}^0_{\kern-0.5pt\R}}}

\newcommand{\ord}{\mathrm{ord}}

\newcommand{\HHNC}{\mbox{\sc HN}_{\kern-1pt\C}}

\newcommand{\partder}[2]{\frac{\partial #1}{\partial #2}}

\def\prem{\mathrm{prem}\,}

\def\Reg{\mathrm{Reg}\,}
\def\Div{\mathrm{Div}\,}

\def\supp{\mathrm{supp}\,}

\def\mF{{\mathcal F}}
\def\mL{{\mathcal L}}
\def\mM{{\mathcal M}}
\def\mK{{\mathcal K}}
\def\mHom{{\mathcal Hom}}
\def\mH{{\mathcal H}}
\def\reg{\mathrm{reg}\,}

\newcommand{\ud}{\mathrm{d}}

\def\rk{\mathrm{rk}\,}

\def\im{\mathrm{im}\,}

\def\FPk{{\rm FP}_{\kern-1pt k}}
\def\Pk{{\rm P}_{\kern-1pt k}}
\def\NPk{{\rm NP}_{\kern-2pt k}}

\newcommand\CPpar[1]{{\rm \#P}_{\kern-2pt #1}}

\def\HNk{\textsc{HN}_{\kern-1pt k}}
\def\DIMk{{\mbox{\sc Dim}_{\kern-0.5pt k}}}

\newcommand{\comment}[1]{}

\def\mC{{\mathcal C}}

\def\mG{{\mathcal G}}

\def\H{\mathbb{H}}

\def\Hdr{H_{\rm{dR}}}
\def\tot{{\rm tot}}
\def\mU{{\mathcal U}}

\title{Castelnuovo-Mumford Regularity and Computing the de Rham
Cohomology of Smooth Projective Varieties}
\author{Peter Scheiblechner\footnote{Partially supported by DFG grant SCHE 1639/1-1.}\\
Hausdorff Center for Mathematics, \\
Endenicher Allee 62, 53115 Bonn, Germany\\
peter.scheiblechner@hcm.uni-bonn.de
}

\date{}

\begin{document}

\maketitle

\begin{abstract}
We describe a parallel polynomial time algorithm for computing the topological
Betti numbers of a smooth complex projective variety $X$.
It is the first single exponential time algorithm for computing the Betti
numbers of a significant class of complex varieties of arbitrary dimension.
Our main theoretical result is that
the Castelnuovo-Mumford regularity of the sheaf of differential
$p$-forms on $X$ is bounded by~$p(em+1)D$, where $e$, $m$, and $D$ are the
maximal codimension, dimension, and degree, respectively, 
of all irreducible components of $X$. 
It follows that, for
a union $V$ of generic hyperplane sections in $X$, the algebraic de
Rham cohomology of $X\setminus V$ is described by
differential forms with poles along $V$ of single exponential order.
By covering $X$ with sets of this type and using a \v{C}ech process, this
yields a similar description of the de Rham cohomology of~$X$,
which allows its efficient computation.
Furthermore, we give a parallel polynomial time algorithm for testing whether a
projective variety is smooth.
\end{abstract}

\noindent {\bf Mathematics Subject Classification (2010)} 68Q25, 14Q20, 14F40, 68W30


\section{Introduction}
A long standing open problem in algorithmic real algebraic geometry is to
construct a single exponential time algorithm for computing the Betti numbers
of semialgebraic sets (for an overview see~\cite{basu:08}). The best
result in this direction is in~\cite{basu:06}
saying that for fixed $\ell$ one can compute the first $\ell$ Betti numbers
of a semialgebraic set in single exponential time. 

In the complex setting one approach for computing Betti numbers is to
compute the algebraic de Rham cohomology. A result of Grothendieck~\cite{gro:66}
states that the de Rham cohomology of a smooth complex variety is canonically
isomorphic to the singular cohomology. An algorithm based on
$\mathcal{D}$-modules for computing the de Rham cohomology of the complement of
a complex affine variety is given
in~\cite{ota:98,wal:00}. This algorithm is used in \cite{wal:00b}
to compute the cohomology of a projective variety.
However, these algorithms are not analyzed with regard to their complexity.
Furthermore, they use Gr\"{o}bner basis computations in a non-commutative
setting, so that a good worst-case complexity is not to be expected. Indeed,
already in the commutative case, computing Gr\"{o}bner bases is exponential 
space complete~\cite{may:82,may:97}.

Via the well known Hodge decomposition, the de Rham cohomology of a smooth
projective variety $X$ is related to the sheaf cohomologies of the sheaves of
differential forms on $X$. Algorithms for computing sheaf cohomology are
described in~\cite{vasc:98,smi:00,eis:03} and implemented in
Macaulay2~\cite{M2}, see also~\cite{eis:01}.

A parallel polynomial time algorithm for counting the connected components
(i.e., computing the zeroth Betti number) of a complex affine variety
is given in~\cite{bus:09}. Although this problem can also be solved by applying
the corresponding real algorithms, the algorithm of~\cite{bus:09} is the first
one using the field structure of $\C$ only. It also extends to counting the
irreducible components. In~\cite{sch:10} it is described
how one can compute equations for the components in parallel polynomial time.

Concerning lower bounds it is shown in~\cite{sch:07} that it is $\PSPACE$-hard
to compute some fixed Betti number of a complex affine or projective variety
given over the integers.
Note that the varieties constructed in this reduction are highly singular. We
do not know of any lower bound result for the problem of testing whether a
variety is smooth.

\subsection{Main Result}
In this paper we describe an algorithm for computing the algebraic de Rham
cohomology of a smooth projective variety running in parallel polynomial time.
It is based on the same techniques as the algorithm in~\cite{bus:09} using
squarefree regular chains. Namely, by applying the algorithm
of Sz{\'a}nt{\'o}~\cite{sza:97,sza:99} we can construct a linear system of
equations describing the ideal (up to a given degree) of an affine variety.
This allows to compute a linear system describing the vanishing of differential
forms of some degree. 
Given a smooth projective variety $X\subseteq\proj^n$
of dimension $m$ and generic hyperplane sections $H_0,\ldots,H_m\subseteq X$
with $H_0\cap\cdots\cap H_m=\emptyset$, their complements $U_i=X\setminus H_i$
form an open affine cover of~$X$. Under the additional assumption that the
hypersurface $H_0\cup\cdots\cup H_m$ has normal crossings, we are able to
compute the cohomologies of the affine patches
$U_{i_0\cdots i_q}=U_{i_0}\cap\cdots\cap U_{i_q}$
and by a \v{C}ech process also the cohomology of $X$.

More precisely, let $D$ be the maximal degree and $e$ the maximal codimension of all irreducible components of $X$, and choose $s\ge m(em+1)D$. 
Now consider the double complex 
$$
K^{p,q}:=\bigoplus_{i_0<\cdots<i_q} \G(X,\Om_X^p((s+p)(H_{i_0}\cup\cdots\cup H_{i_q})))
$$
together with the \v{C}ech differential $\d\colon K^{p,q}\rightarrow K^{p,q+1}$
and the exterior differential $\ud\colon K^{p,q}\rightarrow K^{p+1,q}$.
We show that the de Rham cohomology of $X$ can be computed as the cohomology of the total complex of $K^{\bullet,\bullet}$, i.e., 
$$
\Hdr^\bullet(X)\simeq H^\bullet(\tot^\bullet(K^{\bullet,\bullet})).
$$

To describe the output of the algorithm explicitly, let the $H_i$ be defined by the
linear forms $\ell_i$. The cohomology $\Hdr^k(X)$ is then represented by a basis
consisting of vectors of rational differential forms $\om=(\om_{i_0\cdots i_q})$
of the form
\begin{equation}\label{eq:outputForm}
\om_{i_0\cdots i_q}=\frac{1}{(\ell_{i_0}\cdots\ell_{i_q})^t}
\sum_{0\le j_1<\cdots<j_p\le n}\om_{i_0\cdots i_q}^{j_1\cdots j_p}\ud X_{j_1}
\wedge\cdots\wedge\ud X_{j_p},
\end{equation}
where $p+q=k$, $t\le s+p$, and the polynomials
$\om_{i_0\cdots i_q}^{j_1\cdots j_p}\in\C[X_0,\ldots,X_n]$ are homogeneous of
degree~$t(q+1)-p$.

Furthermore, we show how to test whether a variety $X$ is smooth and how to choose sufficiently generic
generic hyperplanes $H_i$ with normal crossings in parallel polynomial time.
In summary, we prove the following theorem.
\begin{theorem}\label{thm:main}
Given homogeneous polynomials $f_1,\ldots,f_r\in\C[X_0,\ldots,X_n]$ of degree
at most $d$, one can test whether $X:=\mZ(f_1,\ldots,f_r)\subseteq\proj^n$ is
smooth and if so, compute the algebraic de Rham cohomology $\Hdr^\bullet(X)$
in parallel time $(n\log d)^{\Oh(1)}$ and sequential time $d^{\Oh(n^4)}$.
\end{theorem}

As for the necessity of squarefree regular chains in our algorithm
we remark that for smooth varieties there are single exponential bounds for the
degrees in a Gr\"{o}bner basis (cf.~\S\ref{ss:CMReg}).
So perhaps one could replace the squarefree regular chains in our approach
by Gr{\"o}bner bases. But we do not know whether such an algorithm would be
well-parallelizable.

The new aspect of our result compared to the methods of~\cite{ota:98,wal:00,wal:00b}
is the complexity analysis of the algorithm, which is the main purpose of this
paper. We do not expect that our algorithm as stated would yield a practically
efficient implementation. However, it seems conceivable that a variation of the
approach (perhaps using  Gr\"{o}bner bases) could produce a reasonable
implementiation.

\subsection{Castelnuovo-Mumford Regularity}
\label{ss:CMReg}
The main theoretical result of this paper is a bound on the
Castelnuovo-Mum\-ford
regularity of the sheaf of regular differential $p$-forms on a smooth projective
variety $X$. This result allows to bound the degrees of the differential forms
one has to deal with in computing the cohomology of $X$. More precisely, the
regularity yields a bound on the order $t$ in~\eqref{eq:outputForm}, which in
turn determines the degree of the coefficients.

The Castelnuovo-Mumford regularity was defined in~\cite{mum:66} for sheaves
on~$\proj^n$. The definition has been modified in~\cite{eig:84} to apply
to a homogeneous ideal $I$. This notion was related to computational complexity
in~\cite{bast:87} by showing that
the regularity of $I$ equals the maximal degree in a reduced 
Gr\"{o}bner basis of $I$ with respect to the degree reverse lexicographic
order in generic coordinates. In this respect upper bounds on the regularity
of a homogeneous ideal
are of particular interest in computational algebraic geometry and commutative
algebra. For a general ideal, double exponential upper bounds were shown
in~\cite{giu:84} and~\cite{gal:79}. The famous example of~\cite{may:82} shows
that this is essentially best possible. However, there are several
results giving better bounds for the regularity in special cases,
such as~\cite{gru:83,pin:83,lazf:87,ran:90,pee:98,stv:98,kwa:00,gia:06}. 
A nice overview over
these kinds of results is given in~\cite{bamu:93}. This paper also contains
a bound on the regularity of the ideal of a smooth veriety $X$, which is
asymptotic to the product of the degree and the dimension of $X$. A more precise
bound in terms
of the degrees of generators of the ideal of~$X$ is proved
in~\cite{bel:91}. More generally, the authors prove the vanishing of the higher
cohomology of powers of the sufficiently twisted ideal sheaf of~$X$ 
(cf.~Proposition~\ref{prop:vanishing}).
Our bound on the regularity of the sheaf of differential forms on~$X$ is
deduced from this result. We are not aware of any other bounds on the
regularity of a sheaf other than a power of an ideal sheaf.

\subsection*{Acknowledgements}
The author is very grateful to Saugata Basu for being his host, many important
and interesting discussions, and recommending the book~\cite{lazf:04}.
Without him this work wouldn't have been possible. The author also thanks Manoj
Kummini for fruitful discussions about the Castelnuovo-Mumford regularity,
Christian Schnell for a discussion about the cohomology of hypersurfaces, 
Nicolas Perrin and Mart{\'i} Lahoz for useful discussions about exterior powers
of sheaves, as well as the anonymous referees for valuable comments and
suggestions.

\section{Preliminaries}
\subsection{Basic Notations}
\label{ss:basics}
Denote by $\proj^n:=\proj^n(\C)$ the projective space over~$\C$.
A {\em (closed) projective variety} $X\subseteq\proj^n$ is defined as the zero
set
$$
X=\mZ(f_1\ldots,f_r):=\{x\in\proj^n\,|\,f_1(x)=\cdots =f_r(x)=0\}
$$
of homogeneous polynomials $f_1\ldots,f_r\in \C[X_0,\ldots,X_n]$. Note that
$X$ may be reducible. Occasionally we write $\mZ_X(f):=X\cap\mZ(f)$ for $f\in \C[X_0,\ldots,X_n]$.
A {\em quasi-projective variety} is a difference $X\setminus Y$, where~$X$
and~$Y$ are closed projective varieties. The term variety will always mean
quasi-projective variety.
The {\em homogeneous (vanishing) ideal} $I(X)$ of the variety~$X$ is
defined as the ideal generated by the homogeneous polynomials vanishing on~$X$.
The {\em homogeneous coordinate ring} of $X$ is
$\C[X]:=\C[X_0,\ldots,X_n]/I(X)$. 
By the (weak) homogeneous Nullstellensatz we have
$\mZ(f_1,\ldots,f_r)=\emptyset$ iff there exists $N\in\N$ with
$(X_0,\ldots,X_n)^N\subseteq(f_1,\ldots,f_r)$. Its
effective version states that one can choose $N=(n+1)d-n$, when
$\deg f_i\le d$~\cite[Th{\'e}or{\`e}me 3.3]{laz:81}. According to
the affine effective Nullstellensatz, for $f_1,\ldots,f_r\in\C[X_1,\ldots,X_n]$
of degree $\le d$, we have $\mZ(f_1,\ldots,f_r)=\emptyset$ iff there exist
polynomials $g_1,\ldots,g_r$ with $\deg(g_i f_i)\le d^n$ such that
$1=\sum_i g_i f_i$~\cite{brow:87,koll:88,fig:90,jel:05}.

The {\em dimension} $\dim X$ is 
the {\em Krull dimension} of $X$ in the Zariski topology.
A variety all of whose irreducible components have the same
dimension~$m$ is called {\em (m-)equidimensional}. 
The {\em local dimension} $\dim_x X$ at $x\in X$ is defined as the
maximal dimension of all components through $x$.
A {\em hypersurface} of a variety $X$ is a closed subvariety $V$ with
$\dim_x V=\dim_x X-1$ for all $x\in V$.

We often identify $\proj^n\setminus\mZ(X_i)\simeq\C^n$, $0\le i\le n$, via
$(x_0:\cdots: x_n)\mapsto(\frac{x_0}{x_i},\ldots,\hat{\frac{x_i}{x_i}},\ldots,
\frac{x_n}{x_i})$, where as usual $\hat{\ }$ denotes omission. Under this
identification, a homogeneous polynomial $f\in\C[X_0,\ldots,X_n]$ corresponds
to its dehomogenization $f^i$ 
by
setting $X_i:=1$. We thus get a surjection
$^i\colon\C[X_0,\ldots,X_n]\rightarrow\C[X_0,\ldots,\hat{X_i},\ldots,X_n]$,
and the image of $I(X)$ under this map is the vanishing ideal of the affine
variety $X\setminus\mZ(X_i)$.
Now let $f_1,\ldots,f_r$ be homogeneous polynomials 
defining the hypersurfaces $H_1,\ldots,H_r$ of $\proj^n$.
Then we say that the closed variety~$X$ is
{\em scheme-theoretically cut out by the hypersurfaces} $H_1,\ldots,H_r$ iff
for each $i$, the dehomo\-genizations $f_1^i,\ldots,f_r^i$
generate the image of $I(X)$ in $\C[X_0,\ldots,\hat{X_i},\ldots,X_n]$.

For a polynomial $f\in \C[X_1,\ldots,X_n]$ its {\em differential} at $x\in\C^n$
is the linear function $d_x f\colon \C^n\rightarrow\C$ defined by
$d_x f(v):=\sum_i\partder{f}{X_i}(x)v_i$.
The tangent space of the variety $X$ at $x\in X\setminus\mZ(X_i)$
is defined as the vector subspace
$$
T_x X := \{v\in \C^n\,|\,\forall f\in I(X)\;d_xf^i(v)=0\}\subseteq \C^n.
$$
If $X$ is scheme-theoretically cut out by the hypersurfaces defined by the
homogeneous polynomials $f_1,\ldots,f_r$, then
$T_xX=\mZ(d_x f_1^i,\ldots,d_x f_r^i)$.
We have $\dim T_x X\ge \dim_x X$ for all $x\in X$.
We say that $x\in X$ is a {\em smooth} point in $X$ iff
$\dim T_x X=\dim_x X$.
The variety $X$ is {\em smooth} iff all of its points are smooth.

The {\em degree} $\deg X$ of an irreducible closed variety $X$ of dimension
$m$ is defined as the maximal cardinality of $X\cap L$ over all linear
subspaces $L\subseteq\proj^n$ of dimension
$n-m$~\cite[\S5A]{mum:76}.
We define the (cumulative) degree $\deg X$ of a reducible variety $X$ to be the
sum of the degrees of {\em all} irreducible components of $X$.
It follows essentially from B\'{e}zout's Theorem that
if $X$ is defined by polynomials of degree $\le d$, then
$\deg X\le d^n$~\cite{cgh:91}.

\subsection{Coherent Sheaves}
\label{ss:cohSheaves}
Let $X$ be a closed variety in $\proj^n$. 
Then every graded $\C[X]$-module~$M$ gives rise to a sheaf $\tilde{M}$ of
$\Oh_X$-modules on $X$ such that, on a principal open set
$U=X\setminus\mZ(f)$, the sections of~$\tilde{M}$ are given by
$\G(U,\tilde{M})=M_{(f)}$, the degree 0 part of the localization
of $M$ at~$f$.
A sheaf $\mF$ on $X$ is called {\em coherent} iff $\mF=\tilde{M}$
with a finitely generated graded $\C[X]$-module~$M$.

An important example is of course the structure sheaf $\Oh_X=\tilde{\C[X]}$. 
We also define the {\em twisting sheaf}
$\Oh_X(k):=\tilde{\C[X](k)}$ for $k\in \Z$, where $\C[X](k)_d:=\C[X]_{k+d}$.
The twisting sheaf behaves well with respect to direct and inverse images under
maps. In particular, for a closed embedding $i\colon X\hookrightarrow Y$ of
projective varieties
we have $i^*(\Oh_Y(k))\simeq \Oh_X(k)$ and $i_*(\Oh_X(k))\simeq (i_*\Oh_X)(k)$.
Furthermore, we have
$\Oh_{X}(k)\otimes\Oh_{X}(\ell)\simeq\Oh_{X}(k+\ell)$~\cite[II, Proposition 5.12]{hart:77}. 
The sheaf $\Oh_X(1)$ is called the {\em very ample line bundle} on $X$
determined by the embedding $X\hookrightarrow \proj^n$.
For any sheaf of $\Oh_X$-modules $\mF$ on $X$ we define the
{\em twisted} sheaf $\mF(k):=\mF\otimes\Oh_X(k)$ for $k\in\Z$.
The ideal sheaf $\mI_X$ of $X$ is defined as the kernel of the restriction
map $\Oh_{\proj^n}\rightarrow i_*\Oh_X$. 
We have $\mI_X=\tilde{I(X)}$, hence the ideal sheaf is coherent.

Fundamental for us is the sheaf of {\em regular differential forms}, which is defined
as follows. Let $\mU:=\{U_i\,|\,0\le i\le s\}$
be an open affine cover of $X$ (we can take $U_i:=\{X_i\ne 0\}$, but any other
open cover works too).
Denote by $\C[U_i]$ the affine coordinate ring of $U_i$, and let $\Om_{U_i}$ be
the $\C[U_i]$-module of {\em K{\"a}hler differentials} $\Om_{\C[U_i]/\C}$~\cite[Chapter 16]{eise:95}.
Then, $\Om_X$ is defined as the sheaf on~$X$ obtained by glueing
the sheaves on $U_i$ corresponding to $\Om_{U_i}$. This means that 
one determines a section $s$ on an open set $U\subseteq X$ 
by giving for each $i$ 
a section $s_i\in\Om_{U_i\cap U}$ with the property that for all $i,j$ we have
$s_i|U_i\cap U_j \cap U=s_j|U_i\cap U_j \cap U$.
The universal derivations $\ud\colon\C[U_i]\to\Om_{U_i}$ glue together to give
a map $\ud\colon\Oh_X\to\Om_{X}$ of sheaves, which is a derivation on the stalks.
The $p$-fold exterior power $\Om_X^p:=\bigwedge^p\Om_X$ is called the sheaf of
{\em regular (p-)forms}.
The derivation $\ud\colon\Oh_X\to\Om_{X}$ uniquely extends to the {\em exterior
derivative} $\ud\colon\Om_X^p\to\Om_X^{p+1}$ satisfying Leibnitz' rule and
$d\circ d=0$, so that we get the {\em (algebraic) de Rham complex}
$$
\Om_X^\bullet\colon\ 0\longrightarrow\Oh_X\stackrel{\ud}{\longrightarrow}
\Om_X^1\stackrel{\ud}{\longrightarrow}\cdots
\stackrel{\ud}{\longrightarrow}\Om_X^m\longrightarrow 0,
$$
where $m=\dim X$.

For $X=\proj^n$ there is a more explicit description of this complex,
which we give now (cf.~\S6.1 of~\cite{dimc:92}). Denote by~$\L$ the
module of K\"{a}hler differentials
$\Om_{\C[X_0,\ldots,X_n]/\C}$, which is the free module generated by
$\ud X_0,\ldots,\ud X_n$~\cite[Proposition 16.1]{eise:95}. The
universal derivation $\ud\colon\C[X_0,\ldots,X_n]\rightarrow\L$ is given by
$\ud f=\sum_i\partder{f}{X_i}\ud X_i$.
Then, $\L^p:=\bigwedge^p\L$ is the free module generated by
$\ud X_{i_1}\wedge\cdots\wedge\ud X_{i_p}$, $0\le i_1<\cdots <i_p\le n$.
Furthermore, the exterior derivative $\ud\colon\L^p\rightarrow\L^{p+1}$
yields the {\em de Rham complex}
$$
\L^0=\C[X_0,\ldots,X_n]\stackrel{\ud}{\longrightarrow}\L^1
\stackrel{\ud}{\longrightarrow}
\cdots\stackrel{\ud}{\longrightarrow}\L^n\stackrel{\ud}{\longrightarrow}\L^{n+1}
$$
The modules $\L^p$ are graded by setting
$$
\deg(g\ud X_{i_1}\wedge\cdots\wedge\ud X_{i_p}):=\deg(g)+p, \quad
g\in\C[X_0,\ldots,X_n]\quad\textrm{homogeneous}.
$$
Then $\ud$ is a map of degree 0. 
There is another derivation $\D\colon\L^p\rightarrow\L^{p-1}$ of degree~0, 
which can be defined as the {\em contraction with the Euler vector field}
$\sum_i X_i\partder{}{X_i}$. It is uniquely determined by Leibnitz'
rule and the formula $\D(\ud f)=\deg f\cdot f$ for a homogeneous polynomial
$f$, and satisfies $\D(\ud\a)+\ud(\D\a)=\deg\a\cdot \a$ for any homogeneous $\a\in\L^p$.

Now put $M^p:=\ker(\D\colon\L^p\rightarrow\L^{p-1})$. 
One can define the sheaf of differential $p$-forms on $\proj^n$ by setting
$\Om_{\proj^n}^p:=\tilde{M^p}$. 
Hence, for a homogeneous polynomial $f$ of degree $k$, each differential
$p$-form on $\proj^n\setminus\mZ(f)$ is of the form
\begin{equation}
\om=\frac{\a}{f^t}\quad\textrm{with}\quad\deg\a=tk\quad\textrm{and}\quad
\D(\a)=0,
\end{equation}
where $\a\in\L^p$ is homogeneous and $t\in\N$.
By the usual quotient rule one can extend $\ud\colon\L^p\rightarrow\L^{p+1}$ to
localizations.
Then one easily checks that $\ud(M^p_{(f)})\subseteq M^{p+1}_{(f)}$ for
homogeneous $f\in\C[X_0,\ldots,X_n]$.
This defines the exterior derivative
$\ud\colon\Om_{\proj^n}^p\rightarrow\Om_{\proj^n}^{p+1}$ on the sheaf level.

A sheaf $\mF$ on $X$ is said to be {\em locally free} iff it is locally
isomorphic to a direct sum of copies of $\Oh_X$. The local rank of $\mF$
is the number of copies of the structure sheaf needed, which is a locally
constant function. 
If $X$ is smooth and $m$-equidimensional, then $\Om_X^p$ is locally free of rank
$\binom{m}{p}$.

The isomorphism classes of locally free sheaves on $X$ of rank $k$ are in
one-to-one correspondence with those of vector bundles over $X$ of rank
$k$~\cite[II, Exercise~5.18]{hart:77}. A locally free sheaf of rank 1 is called
an {\em invertible sheaf} or {\em line bundle}. The tensor product
$\mL\otimes\mM$ of two line bundles $\mL,\mM$ is also a line bundle.
For any line bundle $\mL$, its {\em dual} sheaf $\check{\mL}:=\mHom(\mL,\Oh_X)$
is another line bundle satisfying
$\mL\otimes\check{\mL}\simeq\Oh_X$~\cite[II, Proposition 6.12]{hart:77}.
In particular, the sheaves $\Oh_X(k)$ defined above are line bundles, and we
have $\Oh_X(k)\check{\ }=\Oh_X(-k)$.


\subsection{Divisors and Line Bundles}
\label{ss:divisors}
Let $X$ be a smooth variety. A divisor on $X$
is an element of the free abelian group $\Div X$ generated by the
irreducible hypersurfaces of $X$. 
This means that each $D\in\Div X$ is a formal linear combination $D=\sum_i m_i V_i$,
where $m_i\in\Z$ and $V_i\subseteq X$ are irreducible hypersurfaces. The
{\em support} of $D$ is defined as $\supp D:=\bigcup_{m_i\ne 0} V_i$.
Each $D\in\Div X$
defines a line bundle $\Oh_X(D)$, which is a subsheaf of the sheaf $\mK$ {\em of total quotient rings of}
$\Oh_X$~\cite[p.~61]{mum:66}.
The stalk $\Oh_X(D)_x$ is $\Oh_{X,x}$, if $x\not\in\supp D$. For
$x\in\supp D$, the stalk is $\prod_i f_i^{-m_i}\Oh_{X,x}$, where 
$f_i$ is a local equation of $V_i$ at $x$.
The rule $D\mapsto \Oh_X(D)$ maps $\Div X$ bijectively onto the invertible
subsheaves of $\mK$, and it maps sums to tensor products~\cite[II, Proposition 6.13]{hart:77}.

Let $X$ be closed and $H$ a hyperplane in $\proj^n$ meeting~$X$
{\em properly}, i.e., $H$ does not contain any irreducible component of $X$, so
that $X\cap H$ is a hypersurface in~$X$.
Then the {\em hyperplane section} $V:=X\cap H$ defines a divisor
$H\cdot X=\sum_i m_i V_i$,
where the $V_i$ are the irreducible components of $V$, and $m_i$ is the
{\em intersection multiplicity} $i(X,H;V_i)$ between $X$ and $H$ along $V_i$,
which can be defined as follows. 
Choose $x\in V_i$ and a (reduced) local equation
$f\in\Oh_{\proj^n,x}$ of $H$. Then $i(X,H;V_i)$ is the {\em order of vanishing}
$\ord_{V_i}(f)$ of $f$ along $V_i$, i.e., the maximal $k\in\N$ such that $f=gh^k$
with some $g$ in $\Oh_{X,x}$, where $h$ is a local equation of $V_i$ at $x$.
Note that $\Oh_{X,x}$ is factorial, since $X$ is smooth.
The line bundle $\Oh_X(H\cdot X)$ is isomorphic to the very ample line bundle
$\Oh_X(1)$.

A hypersurface $V\subseteq X$ is said to have {\em normal crossings} iff 
for each $x\in V$ contained in $k$ irreducible components
$V_1,\ldots,V_k$ of $V$, there exist local equations $f_i\in\Oh_{X,x}$ of $V_i$
around $x$, such that $d_x f_1,\ldots, d_x f_k$ are linearly independent
in the dual $(T_xX)^*$.
Note that the case $k=1$ implies that all irreducible components of $V$ are
smooth. 

\subsection{Sheaf Cohomology}
\label{ss:sheafCoh}
Let $\mF$ be a coherent sheaf and
$\mU:=\{U_i\,|\,0\le i\le s\}$ an open cover of the variety~$X$.
For $0\le q\le s$ and $0\le i_0<\cdots<i_q\le s$ set
$U_{i_0\cdots i_q}:=U_{i_0}\cap\cdots\cap U_{i_q}$.
The {\em \v{C}ech complex} is defined by
$C^{q}:=C^{q}(\mU,\mF):=\bigoplus_{i_0<\cdots<i_q} \mF(U_{i_0\cdots i_q})$,
with the {\em \v{C}ech differential} $\d\colon C^{q}\rightarrow C^{q+1}$
given by
\begin{equation}\label{eq:cechDiff}
(\d(\om))_{i_0\cdots i_{q+1}}:=\sum_{\nu=0}^{q+1}(-1)^\nu 
\om_{{i_0\cdots{\hat i_\nu}\cdots i_{q+1}}}|U_{i_0\cdots i_{q+1}}
\quad\mathrm{for}\quad \om=(\om_{i_0\cdots i_q})\in C^q.
\end{equation}
Then one
easily checks that $\d\circ\d=0$, hence $(C^\bullet,\d)$ is indeed a complex.
Its cohomology $H^i(\mU,\mF):=H^i(C^\bullet,\d)$ is called the $i$-th 
{\em \v{C}ech cohomology} of~$\mF$ with respect to~$\mU$. 
The $i$-th {\em \v{C}ech cohomology} (or {\em sheaf cohomology}) of $\mF$ is
defined as the direct limit over all open covers $\mU$ of $X$, directed by
refinements.
A sheaf~$\mF$ on $X$ is called {\em acyclic} iff 
$H^i(X,\mF)=0$ for all $i>0$.

A cover $\mU$ of $X$ is called a {\em Leray cover} for $\mF$ iff
$\mF$ is acyclic on $U_{i_0\cdots i_{q}}$ for all $i_0<\cdots<i_q$.
Leray's Theorem states that in this case we have
$H^i(X,\mF)=H^i(\mU,\mF)$~\cite[III, Exercise~4.11]{hart:77}.
Since $\mF$ is a coherent sheaf, this is true for any {\em affine}
cover~\cite[III, Theorem~3.5]{hart:77}.
It easily follows that for a morphism $f\colon X\rightarrow Y$ there is a
natural isomorphism $H^i(X,\mF)\simeq H^i(Y,f_*\mF)$, where $f_*\mF$ denotes the
direct image of the sheaf $\mF$ under $f$~\cite[III, Exercise\ 4.1]{hart:77}.

\subsection{Hypercohomology and de Rham Cohomology}
\label{ss:doubleComplexHypercoh}
The material in this section is explained e.g.\ in~\cite{grh:78}.
Let $X$ be a smooth variety and consider a complex of coherent sheaves
$(\mF^\bullet,\ud)$ on~$X$ with $\mF^p=0$ for $p<0$. Then, for an open cover
$\mU$, the \v{C}ech
complexes $C^\bullet(\mU,\mF^p)$ as defined
in~\S\ref{ss:sheafCoh} fit together to the {\em \v{C}ech double complex}
$C^{\bullet,\bullet}:=C^{\bullet,\bullet}(\mU,\mF^{\bullet})$ by setting
$$
C^{p,q}(\mU,\mF^{\bullet})=\bigoplus_{i_0<\cdots<i_q} \mF^p(U_{i_0\cdots i_q})
\quad\text{for all}\quad p,q\ge 0.
$$
The two differentials are the one induced by the differential $\ud$ of $\mF$
and the \v{C}ech differential $\d$ defined by~\eqref{eq:cechDiff}.
Denote by
$\H^\bullet(\mU,\mF^\bullet):=H^\bullet(\tot^\bullet(C^{\bullet,\bullet}))$
the cohomology of the total complex of $C^{\bullet,\bullet}(\mU,\mF^{\bullet})$.
Then the {\em hypercohomology} $\H^i(X,\mF^\bullet)$ of the complex of sheaves
$\mF^\bullet$ is
defined  as the direct limit of $\H^\bullet(\mU,\mF^\bullet)$ over all open
covers $\mU$ of $X$, directed by refinement.
As for any double complex~\cite[\S2.4]{mcc:85},
there are two spectral sequences
\begin{align*}
'E_2^{p,q}=H_\ud^p(H^q(X,\mF^\bullet)) & \quad\Rightarrow\quad \H^{p+q}(X,\mF^\bullet)\quad\mathrm{and}\\
''E_2^{p,q}=H^q(X,\mH^p(\mF^\bullet)) & \quad\Rightarrow\quad \H^{p+q}(X,\mF^\bullet),
\end{align*}
where the {\em cohomology sheaf} is defined by
$$
\mH^p(\mF^\bullet):=\ker(\ud\colon\mF^p\rightarrow\mF^{p+1})/
\im(\ud\colon\mF^{p-1}\rightarrow\mF^p).
$$
The first spectral sequence implies that if all the sheaves $\mF^p$ are
acyclic, then $\H^\bullet(X,\mF^\bullet)=H^\bullet(\G(X,\mF^\bullet))$ is
the cohomology of the complex of global sections.
Similarly as for sheaf cohomology we have
$\H^i(X,\mF^\bullet)\simeq\H^\bullet(\mU,\mF^\bullet)$, if $\mU$ is a Leray
cover for all $\mF^p$.

A map of complexes of sheaves $f\colon\mF^\bullet\rightarrow\mG^\bullet$ is
called a {\em quasi-isomorphism} iff it induces an isomorphism
$\mH^\bullet(\mF^\bullet)\rightarrow\mH^\bullet(\mG^\bullet)$. By comparing the
second spectral sequences of the hypercohomologies of $\mF^\bullet$ and
$\mG^\bullet$ it follows that $f$ induces an isomorphism
$\H^{\bullet}(X,\mF^\bullet)\stackrel{\sim}{\rightarrow}\H^{\bullet}(X,\mG^\bullet)$.

\comment{
Now let $m:=\dim X$ and consider the algebraic de Rham complex
$$
\Om_X^\bullet\colon\ 0\longrightarrow\Oh_X\stackrel{\ud}{\longrightarrow}
\Om_X^1\stackrel{\ud}{\longrightarrow}\cdots
\stackrel{\ud}{\longrightarrow}\Om_X^m\longrightarrow 0
$$
of regular differential forms on $X$ together with the exterior derivative.
}
The {\em algebraic de Rham cohomology} of $X$ is defined as the hypercohomology
$$
\Hdr^\bullet(X):=\H^\bullet(X,\Om_X^\bullet)
$$
of the de Rham complex.
If $X$ is affine, then $H^\bullet_{\rm{dR}}(X)$ can be computed by taking the
cohomology of the complex $\G(X,\Om^\bullet_X)$ of global sections,
since all~$\Om^p_X$ are acyclic.

\subsection{Computational Model}
Our model of computation is that of algebraic circuits over $\C$,
cf.~\cite{vzg:86,bucu:05b}.
We set $\C^\infty:=\bigsqcup_{n\in\N}\C^n$. The \textit{size} of an algebraic
circuit $\mC$
is the number of nodes of $\mC$, and its \textit{depth} is the maximal length
of a path from an input to an output node. We say that a function
$f\colon \C^\infty\rightarrow \C^\infty$ can be computed
\textit{in parallel time~$d(n)$ and sequential time $s(n)$} iff there exists a
polynomial-time uniform family of algebraic circuits $(\mC_n)_{n\in\N}$
over~$\C$ of size~$s(n)$ and depth $d(n)$ such that~$\mC_n$ computes $f|\C^n$.

\subsection{Efficient Parallel Linear Algebra}
\label{ss:linAlg}
We use differential forms to reduce our problem to linear algebra, for which
efficient parallel algorithms exist. 
In particular, we need to be able to solve the following problems: 
\begin{enumerate}
\item Given $A\in\C^{n\times m}$ and $b\in\C^n$, decide whether the linear
system of equations $Ax=b$ has a solution and if so, compute one.
\item Compute a basis of the kernel of a matrix $A\in\C^{n\times m}$.
\item Compute a basis of the image of a matrix $A\in\C^{n\times m}$.
\item Given a linear subspace $V\subseteq\C^n$ in terms of a basis, and given
linearly independent $v_1,\ldots,v_i\in V$, extend them to a basis of $V$.
\end{enumerate}
These problems are easily reduced to inverting a regular square-matrix
(thus to computing the characteristic polynomial) and
computing the rank of a matrix. For instance, the last problem boils down to
rank computations as follows. Let $b_1,\ldots,b_m\in V$ be the given basis.
Set $B:=(v_1,\ldots,v_i)$. For all $j=1,2,\ldots,m$ do: if
$\rk(B,b_j)>\rk B$ then append $b_j$ to $B$.

Mulmuley~\cite{mul:87} has reduced the problem of computing the rank to the
computation of the characteristic polynomial of a matrix. Since we need his
construction, we describe it here.
Let $A\in\C^{m\times m'}$ be a matrix. Then 
$$
\rk\left(\begin{array}{cc}
0 & A\\
A^T & 0
\end{array}\right)=2\,\rk A,
$$
so we can assume $m=m'$. Introducing the additional variable $T$, define the
diagonal matrix $X:=\mathrm{diag}(1,T,\ldots,T^{m-1})$, and
consider the characteristic polynomial $p_A(Z)$ of $XA$ over the field $\C(T)$,
$p_A(Z):=\det(XA-ZI)$. Then the rank of $A$ equals $m-s$,
where $s$ is the  maximal integer with $Z^s|p_A(Z)$. We will call $p_A(Z)$ the
{\em Mulmuley polynomial} of $A$.

The characteristic polynomial of an $m\times m$ matrix can be computed in
parallel (sequential) time $\Oh(\log^2 m)$ ($m^{\Oh(1)}$) with the algorithm
of~\cite{ber:84}.
If the matrix has polynomial entries of degree $d$ in $n$ variables, then
the Berkowitz algorithm can be implemented in parallel (sequential) time
$\Oh(n\log m\log(md))$ ($(md)^{\Oh(n)}$)~\cite{sch:07a}.

\section{Castelnuovo-Mumford Regularity}\label{se:CMRegularity}
A nice exposition about various versions of Castelnuovo-Mumford regularity and
vanishing results is contained in the book~\cite{lazf:04}.
Let $X\subseteq\proj^n$ be a smooth closed subvariety.
Recall from~\S\ref{ss:cohSheaves} that $\Oh_X(1)$ denotes the very ample line
bundle on~$X$ determined by the embedding $X\hookrightarrow \proj^n$, and that
for a coherent sheaf~$\mF$ on $X$ we put $\mF(k):=\mF\otimes\Oh_X(k)$. The
following definition is due to~\cite{mum:66} building on ideas of
Castelnuovo.
\begin{definition}
The coherent sheaf $\mF$ on $X$ is called {\em k-regular} iff
\begin{equation}
H^i\left(X,\mF(k-i)\right)=0\quad\textrm{for all}\quad i>0.
\end{equation}
The {\em Castelnuovo-Mumford regularity} $\reg(\mF)$ of $\mF$ is defined as the
infimum over all $k\in\Z$ such that $\mF$ is $k$-regular.
\end{definition}

\begin{remark}\label{rem:shortExactSequ}
\begin{enumerate}
\item[(i)] A fundamental result of~\cite{mum:66} is that if $\mF$ is
$k$-regular, then $\mF$ is $\ell$-regular for all $\ell\ge k$.
\item[(ii)] Let
$$
0\rightarrow\mF\rightarrow\mG\rightarrow\mH\rightarrow 0
$$
be a short exact sequence of coherent sheaves on $X$.
The long exact cohomology sequence shows that
$$
\reg(\mH)\le\max\{\reg(\mF)-1,\reg(\mG)\}\quad\text{and}
$$
$$
\reg(\mG)\le\max\{\reg(\mF),\reg(\mH)\}.
$$
\item[(iii)] Let $i\colon X\hookrightarrow \proj^n$ be the closed embedding, and
$\mF$ a coherent sheaf on $X$. The projection formula shows
$i_*\left(\mF\otimes i^*\Oh_{\proj^n}(t)\right)=i_*\mF\otimes_{\Oh_{\proj^n}}\Oh_{\proj^n}(t)$
for all $t\in\Z$~\cite[II, Exercise 5.1]{hart:77}, hence
$$
\reg(\mF)=\reg(i_*\mF).
$$
This is usually used as an argument that one can restrict to the case $X=\proj^n$.
However, we are dealing with exterior powers and thus need the more general
situation, since a direct image of a locally free sheaf is not locally free,
and reasonable formulas for exterior powers hold for locally free sheaves only
(in particular, Corollary~\ref{cor:regExterior}).
\item[(iv)] Let $X\subseteq\proj^n$ be a subscheme of dimension $m$ with ideal
sheaf $\mI=\mI_X$, and let $k>0$. Then $\mI$ is $k$-regular if and only if
$H^i\left(\proj^n,\mI(k-i)\right)=0$ for all $0<i\le m+1$~\cite[Example 1.8.29]{lazf:04}.
This follows from the short exact sequence $0\to\mI\to\Oh_{\proj^n}\to i_*\Oh_X\to 0$
using that $H^i(X,\mF)=0$ for all $i>m$ and any coherent sheaf~$\mF$~\cite[III, Theorem 2.7]{hart:77}.
\end{enumerate}
\end{remark}
\begin{example}\label{ex:regPn}
\begin{enumerate}
\item[(i)] 
Theorem 5.1 in Chapter III of~\cite{hart:77} shows 
$\reg(\Oh_{\proj^n})=0$.
\item[(ii)] The structure sheaf $\Oh_X$ of a hypersurface $X\subseteq\proj^n$
of degree $D$ has regularity $D-1$. This follows from the isomorphism
$\mI\simeq\Oh_{\proj^n}(-D)$ for the ideal sheaf $\mI$ of $X$ and the exact
sequence $0\to\mI\to\Oh_{\proj^n}\to i_*\Oh_X\to 0$.
\item[(iii)] Example (i) together with the exact sequence
of~\cite[II, Theorem 8.13]{hart:77} implies
$\reg(\Om_{\proj^n})=2$.
\end{enumerate}
\end{example}
The aim of this section is to prove the following theorem.
\begin{theorem}\label{thm:boundReg}
Let $X\subset\proj^n$ be a smooth closed subvariety of dimension~$m$.
Let $D$ be the maximal degree and $e$ the maximal codimension of all irreducible
components of $X$. Then
\begin{align*}
\reg(\Om_X^p) & \le p(em+1)D\quad\mathrm{for}\quad p>0,\\
\reg(\Oh_X) & \le e(D-1).
\end{align*}
\end{theorem}
\begin{remark}
For $X=\proj^n$ the first claim is false as Example~\ref{ex:regPn} (iii) shows.
\end{remark}
We will reduce this theorem to the following vanishing result of~\cite{bel:91}.
\begin{proposition}~\label{prop:vanishing}
Let $\mI$ be the ideal sheaf of a smooth irreducible closed variety
in~$\proj^n$ of codimension $e$, which
is scheme-theoretically cut out by hypersurfaces of degrees at most~$D$. Then
$$
H^i(\proj^n,\mI^a(k))=0\quad\textrm{for}\quad a\ge 0,\ i>0,\ k\ge (a+e-1)D-n,
$$
where $\mI^a$ denotes the $a$-th power of the ideal sheaf $\mI$.
\end{proposition}
\begin{remark}
In~\cite{bel:91} there is proved a more precise bound in terms of the individual
degrees of the hypersurfaces which cut out $X$, but we do not need this here.
\end{remark}
We also use the following result of~\cite{mum:70}.
\begin{proposition}\label{prop:generation}
Each smooth irreducible closed projective variety of degree $D$ is
scheme-theoretically cut out by hypersurfaces of degree $D$.
\end{proposition}
\begin{corollary}\label{cor:regIdealSheaf}
Let $\mI$ be the ideal sheaf  of a smooth irreducible projective variety~$X$ in $\proj^n$ of degree $D$ and codimension $e$. Then
$$
\reg(\mI^a)\le (a+e-1)D-e+1.
$$
\end{corollary}
\begin{proof}
This follows immediately from Propositions~\ref{prop:vanishing} and~\ref{prop:generation}
together with part (iv) of Remark~\ref{rem:shortExactSequ}
(note that~$\mI^a$ is the ideal sheaf of some
subscheme of the same dimension as $X$). 
\end{proof}

Before we prove Theorem~\ref{thm:boundReg}, let us first gather some basic
properties of regularity. In the following one can always assume~$X$ to be
irreducible. The next two lemmas are a version of a well known technique to
characterize regularity by free resolutions.
\begin{lemma}\label{lem:boundRes}
Let
$$
\mF_N\rightarrow\mF_{N-1}\rightarrow\cdots\rightarrow\mF_0\rightarrow\mF\rightarrow 0
$$
be an exact sequence of coherent sheaves on $X$, where $N+1\ge\dim X=:m$. Then
$$
\reg(\mF)\le\max\{\reg(\mF_0),\reg(\mF_1)-1,\ldots,\reg(\mF_{m-1})-m+1\}.
$$
\end{lemma}
\begin{proof}
This follows easily by chasing through the complex~\cite[Proposition B.1.2]{lazf:04},
taking into account that $H^i(X,\mF)=0$ for all $i>m$ and any coherent sheaf~$\mF$.
Another proof is given in~\cite[Lemma 3.9]{ara:04}.
\end{proof}
For a finite dimensional vector space $V$ and a coherent sheaf $\mF$ on $X$
we denote by $V\otimes\mF$ the sheaf $U\mapsto V\otimes_{\C}\mF(U)$. If
$v_1,\ldots,v_N$ is a basis of $V$, then $V\otimes\mF=\bigoplus_{i=1}^N
v_i\otimes\mF$.
The following lemma is proved as Corollary 3.2 in~\cite{ara:04}. Set 
$\Reg(X):=\max\{1,\reg(\Oh_X)\}$. 
\begin{lemma}\label{lem:freeRes}
Let $\mF$ be a $k$-regular coherent sheaf on $X$. Then there exist finite
dimensional vector spaces $V_i$ and an exact sequence
\begin{equation}\label{eq:freeRes}
\cdots\rightarrow V_i\otimes\Oh_X(-k-iR)\rightarrow\cdots\rightarrow V_1\otimes
\Oh_X(-k-R)\rightarrow V_0\otimes\Oh_X(-k)\rightarrow\mF\rightarrow 0,
\end{equation}
where $R:=\Reg(X)$.
\end{lemma}
Using this we prove a bound on the regularity of tensor products.
\begin{proposition}\label{prop:tensorProd}
Let $\mF,\mG$ be coherent sheaves on $X$, where $\mG$ is locally free, and
denote $m:=\dim X$ and $R:=\Reg(X)$ as above. Then
$$
\reg(\mF\otimes\mG)\le \reg(\mF)+\reg(\mG)+(m-1)(R-1).
$$
\end{proposition}
\begin{proof}
The proof parallels the one of the special case $X=\proj^n$~\cite[Proposition 1.8.9]{lazf:04}.
Let $k:=\reg(\mF)$, and consider the resolution~\eqref{eq:freeRes} of $\mF$,
which exists according to Lemma~\ref{lem:freeRes}. Tensoring with $\mG$
yields
$$
\cdots\rightarrow V_i\otimes\mG(-k-iR)\rightarrow\cdots\rightarrow V_1\otimes
\mG(-k-R)\rightarrow V_0\otimes\mG(-k)\rightarrow\mF\otimes\mG\rightarrow 0.
$$
Since tensoring with a locally free sheaf is an exact functor, this sequence
is exact. Furthermore, $\reg(V_i\otimes\mG(-k-iR))\le k+iR+\reg(\mG)$, hence
$\reg(\mF\otimes\mG)\le k+\reg(\mG)+(m-1)(R-1)$ by Lemma~\ref{lem:boundRes}.
\end{proof}

\begin{corollary}\label{cor:regExterior}
Let $\mF$ be a locally free sheaf on $X$. Then for $p>0$
$$
\reg(\Lambda^p \mF)\le p\cdot\reg(\mF)+(p-1)(m-1)(R-1).
$$
\end{corollary}
\begin{proof}
The same bound for the $p$-th tensor power of $\mF$ clearly follows from
Proposition~\ref{prop:tensorProd}. Since the exterior power is a direct
summand of the tensor power~\cite[III,~\S7.4]{bou:74}, this implies the claim.
\end{proof}

\begin{proposition}\label{prop:reg1Forms}
Let $X\subseteq\proj^n$ be a smooth irreducible subvariety of codimension~$e$
and degree at most $D\ge 2$. Then
$$
\reg(\Om_X)\le (e+1)D-e.
$$
\end{proposition}
\begin{proof}
Denote with $\mI$ the ideal sheaf of $X$ and
let $i\colon X\hookrightarrow \proj^n$ be the inclusion. 
Part (ii) of Remark~\ref{rem:shortExactSequ} applied to the exact sequence
$$
0\rightarrow\mI\rightarrow\Oh_{\proj^n}\rightarrow i_*\Oh_X\rightarrow 0
$$
implies $\reg(\Oh_X)=\reg(i_*\Oh_X)\le\max\{\reg(\mI)-1,\reg(\Oh_{\proj^n})\}$.
Using Corollary~\ref{cor:regIdealSheaf} and Example~\ref{ex:regPn} (i) we conclude 
\begin{equation}\label{eq:regOX}
\reg(\Oh_X)\le eD-e.
\end{equation}
Furthermore, from the exact sequence
$
0\rightarrow\mI^2\rightarrow\mI\rightarrow\mI / \mI^2\rightarrow 0
$
it follows 
\begin{equation}\label{eq:regConormal}
\reg(\mI/\mI^2)\le\max\{\reg(\mI),\reg(\mI^2)-1\}\le (e+1)D-e.
\end{equation}

The last exact sequence we consider is the conormal sequence
\begin{equation}\label{eq:conormal}
0\rightarrow\mI/\mI^2\rightarrow\Oh_X\otimes\Om_{\proj^n}\rightarrow\Om_X\rightarrow 0
\end{equation}
on $X$~\cite[II, Theorem 8.17]{hart:77}. More precisely, the sheaf in the middle is the inverse image 
$$
i^*\Om_{\proj^n}=\Oh_X\otimes_{i^{-1}\Oh_{\proj^n}} i^{-1}\Om_{\proj^n}
$$
of $\Om_{\proj^n}$ under $i$ as an $\Oh_X$-module.
The sheaf on the left is the restriction 
$i^{-1}\mI/\mI^2$, which is automatically an $\Oh_X$-module.

We want to push the sequence~\eqref{eq:conormal} forward to $\proj^n$.
In general, the direct image of sheaves does not commute with stalks. However,
for closed immersions it does: if $\mF$ is a sheaf on $X$, we have
$(i_*\mF)_x=\mF_x$ for $x\in X$, and $0$
otherwise~\cite[II, Exercise 1.19(a)]{hart:77}.
It follows that the functor $i_*$ is exact.

Furthermore, since $\Om_{\proj^n}$ is locally free of finite rank, the projection
formula shows $i_*i^*\Om_{\proj^n}=i_*\Oh_X\otimes_{\Oh_{\proj^n}}\Om_{\proj^n}$.
Since the stalk of $\mI/\mI^2$ at $x\notin X$ vanishes, 
we have $i_*i^{-1}\mI/\mI^2=\mI/\mI^2$.
Thus, applying $i_*$ to~\eqref{eq:conormal} yields the exact sequence
$$
0\rightarrow\mI/\mI^2\rightarrow i_*\Oh_X\otimes\Om_{\proj^n}\rightarrow i_*\Om_X\rightarrow 0.
$$
Hence, $\reg(\Om_X)=\reg(i_*\Om_X)\le\max\{\reg(i_*\Oh_X\otimes\Om_{\proj^n}),\reg(\mI/\mI^2)\}$.
Now Proposition~\ref{prop:tensorProd}, Example~\ref{ex:regPn} (iii),
and~\eqref{eq:regOX} imply
$\reg(i_*\Oh_X\otimes\Om_{\proj^n})\le\reg(\Oh_X)+\reg(\Om_{\proj^n})\le eD-e+2$,
so that $\reg(\Om_X)\le(e+1)D-e$ by~\eqref{eq:regConormal}.
\end{proof}
\begin{remark}
\begin{enumerate}
\item[(i)] In the case $D=1$ we have $\reg(\Om_X)\le 2$.
\item[(ii)] One can check that for a hypersurface $X\subseteq \proj^n$ of degree
$D\ge 3$ one has $\reg(\Om_X)=2D-2$, so that our bound is essentially sharp.
\end{enumerate}
\end{remark}

\noindent{\em Proof of Theorem~\ref{thm:boundReg}.}\quad
The claim can easily be checked for $D=1$.
Also, we can assume $X$ to be irreducible. The claim for $p=0$ is~\eqref{eq:regOX}.
For the case $p\ge 1$, Proposition~\ref{prop:reg1Forms} and
Corollary~\ref{cor:regExterior} imply
\begin{align*}
\reg(\Om_X^p) & \le p((e+1)D-e)+(p-1)(m-1)(e(D-1)-1)\\
 & < p\left((e+1)D-e+(m-1)e(D-1)\right)\\
 & < p(em+1)D. \text{\makebox[239pt][r]{\qedsymbol}}
\end{align*}

\section{Cohomology of Hypersurface Complements}
\subsection{Theory}
Let $X\subseteq\proj^n$ be a smooth closed subvariety.
Using our result on the Castelnuovo-Mumford regularity of the sheaf of
differential forms, one can compute the de Rham cohomology of certain
hypersurface complements in $X$ as the cohomology of finite dimensional
complexes.

To describe these complexes, 
let $H_0,\ldots,H_q\subseteq X$ be hyperplane sections and denote by
$U$ the complement of the hypersurface $V:=\bigcup_\nu H_\nu$ in $X$.
Assume that $V$ has normal crossings (see~\S\ref{ss:divisors}).
We also consider~$V$ as a divisor $V=\sum_\nu H_\nu=\sum_i m_i V_i$, where the
$V_i$ are the irreducible components of~$V$ (cf.~\S\ref{ss:divisors}).
Then, since $\Oh_X(H_\nu)\simeq\Oh_X(1)$, it follows that
\begin{equation}\label{eq:lineBundleHyperplSec}
\Oh_X(V)\simeq\bigotimes_\nu\Oh(H_\nu)\simeq\Oh_X(1)^{\otimes(q+1)}\simeq\Oh_X(q+1).
\end{equation}
Now let $A=\sum_i a_i V_i$ be any divisor with support in $V$, and let
$j\colon U\hookrightarrow X$ be the inclusion. Define the subsheaf
$\Om_X^p(A):=\Om_X^p\otimes\Oh_X(A)$ of $j_*\Om_U^p$, which consists of those
rational differential $p$-forms on $X$, which are regular on $U$ and have poles
(zeros if $a_i<0$) of order $|a_i|$ along $V_i$. Define the sheaves
$$
\mK_X^p(A):=\Om_X^p(A+pV).
$$
Note that $\ud(\mK_X^p(A))\subseteq\mK_X^{p+1}(A)$, so that $\mK_X^\bullet(A)$ 
is in fact a subcomplex of~$j_*\Om_U^\bullet$.
For~$A=V$ it is the zeroth term of the 
{\em polar filtration}~\cite{ded:90,dimc:92}.

The next lemma is the crucial fact that allows us to compute the algebraic de
Rham cohomology of $U$ by a finite dimensional complex. Its proof requires to 
consider holomorphic differential forms.
So let $\Om_{U^{\rm an}}^\bullet$ denote
the complex of holomorphic differential forms on $U^{\rm an}$ regarded as a
complex manifold, and let $\mK_{X^{\rm an}}^\bullet(A)$ be the holomorphic
version of $\mK_{X}^\bullet(A)$.
The following lemma is proved analogously to the corresponding
statement for the logarithmic complex (cf.~\cite{del:70,grh:78,voi:02}). 
The calculation can be found in~\cite[Lemma 4.1]{abg:73}.

\begin{lemma}\label{lem:quasiiso}
Let $A=\sum_i a_i V_i$ be a divisor with $a_i>0$ for all $i$, and assume that
$V$ has normal crossings. Then the inclusion
$\mK_{X^{\rm an}}^\bullet(A)\hookrightarrow j_*\Om_{U^{\rm an}}^\bullet$ is a
quasi-isomorphism.
\end{lemma}

The following is the main result of this section and the key for our algorithm.
\begin{theorem}\label{thm:cohOpenGlobalSec}
Let $X\subseteq\proj^n$ be a smooth closed subvariety with dimension at most
$m\ge 1$.
Let~$D\ge 2$ and $e$ be upper bounds on the degree  and the codimension of all
irreducible components of~$X$. Let $H_0,\ldots,H_q$ be hyperplane sections of
$X$ such that $V=H_0\cup\ldots\cup H_q$ has normal crossings, and denote
$U:=X\setminus V$. For $s\in\N$ set $K_s^\bullet:=\G(X,\mK_X^\bullet(sV))$. Then
we have
$$
\Hdr^\bullet(U)\simeq H^\bullet(K_s^\bullet)\quad\mathrm{for}\quad
s\ge m(em+1)D.
$$
\end{theorem}
\begin{proof}
It follows from Lemma~\ref{lem:quasiiso} that
\begin{equation}\label{eq:isoHypercoh}
\H^\bullet(X^{\rm an},\mK_{X^{\rm an}}^\bullet(sV))\simeq\H^\bullet(X^{\rm an},j_*\Om_{U^{\rm an}}^\bullet).
\end{equation}
Since $\mK_{X}^\bullet(sV)$ and $j_*\Om_{U}^\bullet$ are coherent sheaves on $X$,
by GAGA~\cite{ser:56} the hypercohomologies in~\eqref{eq:isoHypercoh} can be
replaced by their algebraic versions. But we have
$H^i(X,j_*\Om_{U}^p)=H^i(U,\Om_{U}^p)=0$ for $i>0$ since $U$ is affine
(see~\S\ref{ss:cohSheaves}). Hence, the right hypercohomology
in~\eqref{eq:isoHypercoh} is $\Hdr^\bullet(U)$. 

On the other hand, 
Theorem~\ref{thm:boundReg} implies
$s\ge\reg(\Om_X^p)$ for all $0\le p\le m$, thus
$$
H^i(X,\Om_X^p((s+p)V))=H^i(X,\Om_X^p((s+p)(q+1)))=0\quad\textrm{for all}\quad i>0
$$
(use~\eqref{eq:lineBundleHyperplSec}). It follows that the left side
of~\eqref{eq:isoHypercoh} is $H^\bullet(K_s^\bullet)$ as claimed.
\end{proof}

\subsection{Computation}
\label{ss:compCohHypCompl}
We adopt the notations and assumptions of the last section.
We choose $s$ according to Theorem~\ref{thm:cohOpenGlobalSec} and set $K^\bullet:=K_s^\bullet$.
In this section we
describe this finite dimensional complex more explicitly and show
how to compute its cohomology.

Let $H_\nu=\mZ_X(\ell_\nu)$ with linear forms $\ell_\nu$, $0\le\nu\le q$.
We can assume w.l.o.g.\ that $\ell_0=X_0$. With $f:=X_0\ell_1\cdots \ell_q$ we
have $V=\mZ_X(f)$ and $U=X\setminus V$. On the ambient space we define
$\tilde{V}:=\mZ(f)$ and $\tilde{U}:=\proj^n\setminus\tilde{V}$.
Recall from~\S\ref{ss:cohSheaves} that each $p$-form on $\tilde{U}$
is given by
\begin{equation}\label{eq:diffForm}
\om=\frac{\a}{f^t}\quad\textrm{with}\quad\deg\a=t(q+1)\quad\textrm{and}\quad
\D(\a)=0,
\end{equation}
where $\a$ is a homogeneous $p$-form on $\C^{n+1}$ and $\D$ denotes contraction
with the Euler vector field.
Note that for fixed $t$ such forms are precisely the global sections 
of the sheaf $\Om_{\proj^n}^p(t\tilde{V})$.
We have the following
\begin{lemma}\label{lem:restrSurj}
With $s\ge m(em+1)D$ the restriction map
$$
\G(\proj^n,\Om_{\proj^n}^p(t\tilde{V}))\longrightarrow\G(X,\Om_{X}^p(tV))
$$
is surjective for all $t\ge s$ and $p\ge 0$.
\end{lemma}
\begin{proof}
Let $i\colon X\hookrightarrow\proj^n$ be the inclusion and consider the exact
sequence
\begin{equation}\label{eq:alpha}
0\to\mI\otimes\Om_{\proj^n}^p\to\Om_{\proj^n}^p\stackrel{\a}{\to}
i_*\Oh_X\otimes\Om_{\proj^n}^p\to 0 
\end{equation}
on $\proj^n$, as well as
\begin{equation}\label{eq:beta}
0\to\ker\b\to i^*\Om_{\proj^n}^p\stackrel{\b}{\to}\Om_X^p\to 0
\end{equation}
on $X$.
The restriction map from the lemma coincides with the composition of
$$
H^0(\proj^n,\Om_{\proj^n}^p(t(q+1)))\stackrel{\a^*}{\longrightarrow}
H^0(\proj^n,i_*\Oh_X\otimes\Om_{\proj^n}^p(t(q+1)))
\simeq H^0(X,i^*\Om_{\proj^n}^p(t(q+1)))
$$
with
$$
H^0(X,i^*\Om_{\proj^n}^p(t(q+1)))\stackrel{\b^*}{\longrightarrow}H^0(X,\Om_{X}^p(t(q+1))).
$$
We show that $\a^*$ and $\b^*$ are surjective for $t\ge s$. For $\a^*$ this
follows immediately by twisting the short exact sequence~\eqref{eq:alpha} and
considering the induced long exact sequence of cohomology, taking into account
that $H^1(\proj^n,\mI\otimes\Om_{\proj^n}^p(t(q+1)))=0$, since
$$
\reg(\mI\otimes\Om_{\proj^n}^p)\le\reg(\mI)+p\cdot\reg(\Om_{\proj^n})\le e(D-1)+1+2p
\le s+1.
$$
We can prove that $\b^*$ is surjective by the same method, once we bound the
regularity of $\ker\b$. By~\cite[II, Exercise 5.16]{hart:77} the exact sequence
$$
0\rightarrow\mI/\mI^2\to i^*\Om_{\proj^n}\to \Om_X\rightarrow 0
$$
induces a filtration of locally free sheaves on $X$
$$
\bigwedge^p (i^*\Om_{\proj^n})=\mF^0\supseteq\mF^1\supseteq\cdots
\supseteq\mF^p\supseteq\mF^{p+1}=0
$$
and for all $0\le j\le p$ an exact sequence
\begin{equation}\label{eq:quotients}
0\to\mF^{j+1}\to\mF^j\to\bigwedge^j\mI/\mI^2\otimes \Om_X^{p-j}\to 0
\end{equation}
In particular, for $j=0$ this sequence
coincides with~\eqref{eq:beta}, since
$\bigwedge^p (i^*\Om_{\proj^n})=i^* \Om_{\proj^n}^p$~\cite[II, Exercise~5.16]{hart:77},
so that $\ker\b\simeq\mF^1$.

The same calculation as in the proof of Theorem~\ref{thm:boundReg} using
Proposition~\ref{prop:reg1Forms} and Corollary~\ref{cor:regExterior} shows
$$
\reg\big(\bigwedge^j\mI/\mI^2\otimes \Om_X^{p-j}\big)\le p(em+1)D\le s.
$$
Furthermore, by Remark~\ref{rem:shortExactSequ} (ii) the sequence~\eqref{eq:quotients} yields
$$
\reg(\mF^j)\le\max\{\reg(\mF^{j+1}),s\},
$$
which inductively implies $\reg(\mF^j)\le s$ for all $j$, in particular
$\reg(\ker\b)\le s$.
\end{proof}
It follows from the Lemma
 that each element of  $K^p=\G(X,\Om_{X}^p((s+p)V))$
is the restriction of a form in $\Om^p:=\G(\proj^n,\Om_{\proj^n}^p(t\tilde{V}))$.

We identify $\C^n\simeq\{X_0\ne 0\}\subseteq\proj^n$ and set
$X^0:=X\setminus\mZ(X_0)$. 
As with polynomials one can dehomogenize a homogeneous differential
form $\a$ on $\C^{n+1}$ by setting $X_0=1$ and $\ud X_0=0$ to get a form
$\a^0$ on $\C^n$.
Hence for $\om\in\Om^p$ one gets a regular form $\om^0$ on
$\C^n\setminus\mZ(f^0)$.
Its restriction defines a regular form on the dense open subset
$U=X^0\setminus\mZ(f^0)$ of $X^0$.

We use the algorithm of Sz{\'a}nt{\'o}~\cite{sza:97,sza:99} to compute a decomposition
$I:=I(X^0)=\bigcap_j I_j$, where each $I_j$ is the saturated ideal of a
squarefree regular chain $G_j$. Note that $\mZ(I_j)$ is equidimensional.
We will construct for all $j$ a linear system of equations describing the
identity $\om=0$ on~$\mZ(I_j)$ for $\om\in K^p$.
For simplicity we assume that $I$ is represented by a single
$G=\{g_1,\ldots,g_e\}$.
In the general case one only has to combine all the linear systems
to one large system.

Let $k\in\N$. In~\cite{bus:09} we have constructed a linear system of
equations
\begin{equation}\label{eq:prem}
\prem_k(f,G)=0
\end{equation}
in the coefficients of $f\in\C[X_1,\ldots,X_n]$,
whose solution space is $I_{\le k}$, the set of polynomials of degree
$\le k$ vanishing on $X^0$. 

Sz{\'a}nt{\'o}'s algorithm also yields a polynomial $h$ which is a
non-zerodivisor mod~$I$, and such that the module of differentials on
$X^0\setminus\mZ(h)$ is the free module generated by $m$ of the $\ud X_j$,
where $m=\dim X^0$.
More precisely, let $X_1,\ldots,X_m$ denote the free variables, and
$Y_1,\ldots,Y_e$ the dependent variables, where $m+e=n$.
For a polynomial  $F\in\C[X_1,\ldots,X_m,Y_1,\ldots,Y_e]$ we denote
$\overline{F}:= F \bmod I\in \C[X^0]$.
Then by~\cite[Propositon 3.13]{bus:09} we have
$\Om_{\C[X^0]_h/\C}=\bigoplus_{i=1}^m\C[X^0]_h\ud \overline{X}_i$. Furthermore,
for all $F\in\C[X_1,\ldots,X_m,Y_1,\ldots,Y_e]$
\begin{equation}\label{eq:differential}
\ud\overline{F}=\sum_{i=1}^m\left(\partder{F}{X_i}-\partder{F}{Y}
\left(\partder{g}{Y}\right)^{-1}\partder{g}{X_i}\right)\ud \overline{X}_i,
\end{equation}
where $g:=(g_1,\ldots,g_e)^T$. 
Note that $h$ is a multiple of 
$\det(\partder{g}{Y})$, so that the entries of
$(\partder{g}{Y})^{-1}$ lie in $h^{-1}\C[X_1,\ldots,X_n]$.
Using~\eqref{eq:differential} for the coordinates $Y_j$,
one can write the restriction of a form $\om\in\Om^p$ to
$U\setminus\mZ(h)$ in terms of the free generators of $\Om_{\C[U]_h/\C}^p$,
which are $\ud \overline{X}_{i_1}\wedge\cdots\wedge\ud \overline{X}_{i_p}$,
where $1\le i_1<\cdots<i_p\le m$. 
It follows that $\om=\frac{\om_h}{h(f^0)^t}$, where
$$
\om_h=\sum_{1\le i_1<\cdots<i_p\le m}(\om_h)_{i_1\cdots i_p}\ud 
\overline{X}_{i_1}\wedge\cdots\wedge\ud \overline{X}_{i_p}\in\Om_{\C[X^0]/\C}^p.
$$
Then, since $U\setminus\mZ(h)$ is dense in $U$ and in $X^0$, we have
\begin{eqnarray}\label{eq:formIdentLinSys}
\om=0\ \textrm{on}\ U & \iff & \om_h=0\ \textrm{on}\ U\setminus\mZ(h)\nonumber\\
& \iff &
\forall i_1<\cdots<i_p\colon\ (\om_h)_{i_1\cdots i_p}=0\ \textrm{on}\ U\setminus\mZ(h)\nonumber\\
& \iff &
\forall i_1<\cdots<i_p\colon\ (\om_h)_{i_1\cdots i_p}\in I_{\le k}
\end{eqnarray}
for sufficiently large $k$.

Now we compute the cohomology of $K^\bullet$.
First note that the contraction with the Euler vector field $\D$ (cf.~\S\ref{ss:cohSheaves}) can be easily computed, so that we can
compute a basis for $\Om^p$.
Consider the commutative diagram
$$
\xymatrix@M+3pt{
\Om^p \ar[r]^\ud \ar[d]^\pi & \Om^{p+1} \ar[d]^\pi \\
K^p \ar[r]^\ud & K^{p+1}, \\
}
$$
where $\pi$ is the restriction of forms to $U$. Let
$N^p:=\ker(\pi\colon\Om^p\rightarrow K^p)$.
According to~\eqref{eq:formIdentLinSys} and~\eqref{eq:prem}, $N^p$ is the
solution set of a linear system of equations. Since~$\pi$ is surjective,
we have $K^p\simeq \Om^p/N^p$, thus $K^p$ can be identified with
any complementary subspace of $N^p$ in $\Om^p$.
So we compute a basis of $N^p$ and extend it to a basis of $\Om^p$
to get a basis of $K^p$ via this identification. The differential
$\ud\colon K^p\rightarrow K^{p+1}$ is just the restriction of the differential
$\ud\colon\Om^p\rightarrow\Om^{p+1}$, which we can evaluate efficiently.
Hence we can compute
the matrix of $\ud\colon K^p\rightarrow K^{p+1}$ with respect to the computed
bases of $K^\bullet$.  By computing kernel and image of this matrix and taking
their quotient we get the cohomology of $K^\bullet$. 

\begin{proposition}\label{prop:compCohCompl}
Under the notations and assumptions of Theorem~\ref{thm:cohOpenGlobalSec}, let
$X\subseteq\proj^n$ be given by equations of degree $\le d$.
Then one can compute the cohomology $\Hdr^\bullet(U)$ in parallel time
$(d\log n)^{\Oh(1)}$ and sequential time $d^{\Oh(n^4)}$.
\end{proposition}
\begin{proof}
It remains to analyze the algorithm described above. Let $\d$ denote the maximal
degree of the polynomials in the squarefree regular chain $G$. Then the
system~\eqref{eq:prem}
has asymptotic size $\Oh((nk\d^e)^{n})$ and can be computed in parallel time
$(n\log(k\d))^{\Oh(1)}$ and sequential time $(k\d)^{\Oh(n^4)}$~\cite{bus:09}. 
Since the numerator of each $\om\in\Om^p$ has degree $(s+p)(q+1)$, the dimension
of $\Om^p$ is $\binom{n+1}{p}\binom{(s+p)(q+1)-p+n}{n}=\Oh(s^nn^{\Oh(n)})$. Furthermore, for $\om\in\Om^p$,
the degree of the coefficients of $\om_h$ is bounded by $(s+p)(q+1)+p(e+1)\d$, hence
we must choose the $k$ in~\eqref{eq:formIdentLinSys} of that order. Thus,
$N^p$ is described by a linear system of equations 
of size $\Oh(n^n((s+n)n+n^2\d)^n\d^{en})\le n^{\Oh(n)}s^n\d^{(e+1)n}$. 
Now let $X$ be given by equations of degree~$d$.
According to Theorem~\ref{thm:cohOpenGlobalSec}, we have to choose~$s$ of
order $n^3\deg X\le n^3d^n$. Furthermore, by~\cite{sza:99} we have $\d=d^{\Oh(n^2)}$.
Hence the size of this system is $d^{\Oh(n^4)}$.
The algorithms of \S\ref{ss:linAlg} imply the claimed bounds.
\end{proof}

\section{Patching Cohomologies}\label{se:patching}
Let $X$ be a smooth projective variety of dimension at most $m\ge 1$. Our aim
is to compute the de Rham cohomology of $X$ by way of an open affine cover. 

So let $H_0,\ldots,H_m\subseteq X$ be hyperplane sections
with $H_0\cap\cdots\cap H_m=\emptyset$ and set $U_i:=X\setminus H_i$.
Then $\mU:=\{U_i\,|\,0\le i\le m\}$ is an open affine cover of $X$.
Consider the {\em \v{C}ech double complex} $C^{\bullet,\bullet}:=C^{\bullet,\bullet}(\mU,\Om_X^\bullet)$
as defined in \S\ref{ss:doubleComplexHypercoh}.
Recall that with $U_{i_0\cdots i_q}=U_{i_0}\cap\cdots\cap U_{i_q}$ we have
$C^{p,q}(\mU,\Om_X^{\bullet})=\bigoplus_{i_0<\cdots<i_q}
\Om_X^p(U_{i_0\cdots i_q})$.
Since $\mU$ is a Leray cover for all the sheaves $\Om_X^p$, we have
\begin{lemma}\label{lem:dRCechDblComplex}
$\Hdr^\bullet(X)\simeq\H^\bullet(\mU,\Om_X^\bullet)=H^\bullet(\tot^\bullet(C^{\bullet,\bullet}))$,
where $\tot^\bullet(C^{\bullet,\bullet})$ denotes the total complex associated
to $C^{\bullet,\bullet}$.
\end{lemma}

To compute this cohomology, we replace the infinite dimensional
double complex $C^{\bullet,\bullet}$ by a finite dimensional one,
which is built from the complex of the last section 
for each $U_{i_0\cdots i_q}$.
More precisely, 
let $e,D$ have the meanings of 
Theorem~\ref{thm:cohOpenGlobalSec}, and choose $s\ge m(em+1)D$.
For a hypersurface $V$ in $X$ we denote
$K^p(V):=\G(X,\Om_X^p((s+p)V))$. 
This corresponds
to the complex $K_s^\bullet$ from Theorem~\ref{thm:cohOpenGlobalSec}.
Now we define the double complex 
$$
K^{p,q}:=\bigoplus_{i_0<\cdots<i_q} K^p(H_{i_0}\cup\cdots\cup H_{i_q})
$$
together with the differential $\d\colon K^{p,q}\rightarrow K^{p,q+1}$,
which is the restriction of the \v{C}ech differential~\eqref{eq:cechDiff},
and the exterior differential $\ud\colon K^{p,q}\rightarrow K^{p+1,q}$.
Then $K^{\bullet,\bullet}$ is a subcomplex of $C^{\bullet,\bullet}$.
\begin{lemma}\label{lem:dRDblComplex}
We have
$\Hdr^\bullet(X)\simeq H^\bullet(\tot^\bullet(K^{\bullet,\bullet}))$.
\end{lemma}
\begin{proof}
Clearly, the inclusion
$K^{\bullet,\bullet} \hookrightarrow C^{\bullet,\bullet}$ induces a morphism of
spectral sequences
$''E_r(K^{\bullet,\bullet})\rightarrow\, ''E_r(C^{\bullet,\bullet})$
between the second spectral sequences of these double complexes.
Theorem~\ref{thm:cohOpenGlobalSec} implies that this is an isomorphism
$$
''E_1^{p,q}(K^{\bullet,\bullet})\simeq\bigoplus_{i_0<\cdots<i_q}
\Hdr^p(U_{i_0\cdots i_q})=\,''E_1^{p,q}(C^{\bullet,\bullet})
$$
on the first level of these spectral sequences.
According to~\cite[Theorem~3.5]{mcc:85}, this induces an isomorphism on their
$\infty$-terms and, since the corresponding filtrations are bounded, also on the
cohomologies of the total complexes, 
so $H^\bullet(\tot^\bullet(K^{\bullet,\bullet}))\simeq
H^\bullet(\tot^\bullet(C^{\bullet,\bullet}))$.
Together with Lemma~\ref{lem:dRCechDblComplex} this completes the proof.
\end{proof}

\begin{proposition}\label{prop:cohomology}
Assume that one is given homogeneous polynomials of degree at most $d$ defining
the smooth variety $X\subseteq\proj^n$, and linear forms defining the hyperplane 
sections~$H_0,\ldots,H_m$ such that
$\bigcup_i H_i$ has normal crossings.
Then one can compute $\Hdr(X)$ in parallel time $(d\log n)^{\Oh(1)}$ and
sequential time $d^{\Oh(n^4)}$.
\end{proposition}
\begin{proof}
By Lemma~\ref{lem:dRDblComplex} one has to compute the cohomology of the
total complex $T^k:=\tot^k (K^{\bullet,\bullet})=\bigoplus_{p+q=k}K^{p,q}$
with the differential
$$
d_T\colon T^k\rightarrow T^{k+1},\ (\om_{p,q})_{p+q=k}\mapsto
\left(\ud \om_{p-1,q}+(-1)^p\d\om_{p,q-1}\right)_{p+q=k+1}.
$$
As in \S\ref{ss:compCohHypCompl} one can compute bases for $K^{p,q}$ and hence
for $T^\bullet$ within the claimed bounds. Since the differential $d_T$ is
easily computable and the vector spaces have dimension $d^{\Oh(n^4)}$, the
cohomology of this complex can be computed using the
algorithms of \S\ref{ss:linAlg}.
\end{proof}

We conclude this section with an
\begin{example}
Consider the plane curve $X=\mZ(X_0^3+X_1^3+X_2^3)\subseteq\proj^2$. One easily
checks that $X$ is smooth, and the well known genus formula shows that its
Betti numbers are
$$
b_0(X)=b_2(X)=1,\quad b_1(X)=2.
$$
Furthermore, the hyperplanes $\tilde{H}_0:=\mZ(X_0)$ and $\tilde{H}_1:=\mZ(X_1)$
intersect $X$ transversally, and their hyperplane sections $H_0:=\mZ_X(X_0)$
and $H_1:=\mZ_X(X_1)$ do not intersect. Hence, $U_i:=X\setminus H_i$, $i\in\{0,1\}$,
form an open cover of $X$ satisfying our assumptions.

Our theorems yield the lower bound $s\ge 6$, but let us find the smallest
possible $s$. The bounds of Theorem \ref{thm:boundReg} are $\reg(\Oh_X)\le 2$
and $\reg(\Om_X)\le 6$, but the more precise Proposition~\ref{prop:reg1Forms}
yields $\reg(\Om_X)\le 5$. However, computations with Macaulay2~\cite{M2} give
$H^1(\Oh_X(1))=H^1(\Om_X(1))=0$ and $H^1(\Oh_X)\ne 0$, $H^1(\Om_X)\ne 0$, hence
$\reg(\Oh_X)=\reg(\Om_X)=2$ (since $\dim X=1$). A close look at the proofs shows
that the conclusions of Theorem~\ref{thm:cohOpenGlobalSec} and
Lemma~\ref{lem:dRDblComplex} hold for $s=1$. But Macaulay2 also computes
$\dim H^0(\Om_{\proj^2}(2))=3$ and $\dim H^0(\Om_X(2))=6$, so the restriction map
in Lemma~\ref{lem:restrSurj} cannot be surjective.
In order to get surjectivity, note that $\reg(\mI)=3$, thus the restriction map 
$\G(\proj^n,\Oh_{\proj^n}(t\tilde{H}_0))\to\G(X,\Oh_{X}(tH_0))$ is surjective for
$t\ge 2$. As for the one-forms, we compute $\reg(\mI\otimes\Om_{\proj^2})=\reg(\mI/\mI^2)=5$.
It follows that the corresponding restriction map is surjective for $t\ge 4$, so
we can choose $s=3$.

The double complex we have to consider is
$$
\xymatrix@M+3pt{
K^{0,1} \ar[r]^{\ud^1}  & K^{1,1} \\
K^{0,0} \ar[r]^{\ud^0} \ar[u]^{\d^0} & K^{1,0} \ar[u]^{\d^1}, \\
}
$$
where
$$
K^{0,0}=K^0(H_0)\oplus K^0(H_1)=\G(X,\Oh_X(3H_0))\oplus\G(X,\Oh_X(3H_1)),
$$
$$
K^{1,0}=K^1(H_0)\oplus K^1(H_1)=\G(X,\Om_X(4H_0))\oplus\G(X,\Om_X(4H_1)),
$$
$$
K^{0,1}=K^0(H_0\cup H_1)=\G(X,\Oh_X(3H_0+3H_1)),
$$
$$
K^{1,1}=K^1(H_0\cup H_1)=\G(X,\Om_X(4H_0+4H_1)),
$$
together with the differentials
$$
\d^0(f,g)=g-f,\quad \ud^0(f,g)=(\ud f,\ud g),\quad
\d^1(\om,\eta)=\eta-\om,\quad \ud^1(f)=\ud f.
$$
The dimensions of these vector spaces are
$$
\dim K^{0,0}=\dim K^{0,1}=18,\quad \dim K^{1,0}=\dim K^{1,1}=24.
$$
Put $\Om_{ij}=X_i\ud X_j-X_j\ud X_i$ (for simplicity we write $X_i$ instead of
$\overline{X}_i$). Then, bases for these spaces are given by
\begin{align*}
K^0(H_0): \quad & \frac{1}{X_0^3}
(X_0^3,X_0^2X_1,X_0^2X_2,X_0X_1^2,X_0X_1X_2,X_0X_2^2,X_1^3,X_1^2X_2,X_1X_2^2) \\
K^0(H_1): \quad & \frac{1}{X_1^3} 
(X_0^3,X_0^2X_1,X_0^2X_2,X_0X_1^2,X_0X_1X_2,X_0X_2^2,X_1^3,X_1^2X_2,X_1X_2^2) \\
K^1(H_0): \quad & \frac{1}{X_0^4} 
(X_0^2\Om_{01},X_0X_1\Om_{01},X_0X_2\Om_{01},X_1^2\Om_{01},X_1X_2\Om_{01},X_2^2\Om_{01}, \\
 & \qquad\ X_0^2\Om_{02},X_0X_1\Om_{02},X_0X_2\Om_{02},X_1^2\Om_{02},X_1X_2\Om_{02},X_1X_2\Om_{12}) \\
K^1(H_1): \quad & \frac{1}{X_1^4} 
(X_0^2\Om_{01},X_0X_1\Om_{01},X_0X_2\Om_{01},X_1^2\Om_{01},X_1X_2\Om_{01},X_2^2\Om_{01}, \\
 & \qquad\  X_0^2\Om_{02},X_0X_1\Om_{02},X_0X_2\Om_{02},X_1^2\Om_{02},X_1X_2\Om_{02},X_1X_2\Om_{12}) \\
K^0(H_0\cup H_1): & \\
 \frac{1}{(X_0X_1)^3} &
(X_0^6,X_0^5X_1,X_0^5X_2,X_0^4X_1^2,X_0^4X_1X_2,X_0^4X_2^2,X_0^3X_1^3,X_0^3X_1^2X_2, \\
 & X_0^3X_1X_2^2, X_0^2X_1^4, X_0^2X_1^3X_2,X_0^2X_1^2X_2^2,X_0X_1^5,X_0X_1^4X_2, \\
 & X_0X_1^3X_2^2,X_1^6,X_1^5X_2,X_1^4X_2^2) \\
K^1(H_0\cup H_1): & \\
 \frac{\Om_{01}}{(X_0X_1)^4} &
(X_0^6,X_0^5X_1,X_0^5X_2,X_0^4X_1^2,X_0^4X_1X_2,X_0^4X_2^2,X_0^3X_1^3,X_0^3X_1^2X_2, \\
 & X_0^3X_1X_2^2, X_0^2X_1^4, X_0^2X_1^3X_2,X_0^2X_1^2X_2^2,X_0X_1^5,X_0X_1^4X_2, \\
 & X_0X_1^3X_2^2,X_1^6,X_1^5X_2,X_1^4X_2^2), \\
 \frac{\Om_{02}}{(X_0X_1)^4} &
(X_0^6,X_0^5X_1,X_0^5X_2,X_0^4X_1^2,X_0^4X_1X_2,X_0^3X_1^2X_2) \\
\end{align*}

The corresponding total complex $T^\bullet=\tot^\bullet(K^{\bullet,\bullet})$ is
$$
T^0=K^{0,0}\stackrel{\ud^{\tot,0}}{\longrightarrow} T^1=K^{0,1}\oplus K^{1,0}
\stackrel{\ud^{\tot,1}}{\longrightarrow} T^2=K^{1,1},
$$
where
$$
\ud^{\tot,0}=(\d^0,\ud^0),\quad \ud^{\tot,1}=\ud^1-\d^1.
$$
The matrix of $\ud^{\tot,0}$ with respect to the given bases is
$$
{\tiny
\left(
\begin{array}{cccccccccccccccccc}
 0 & 0 & 0 & 0 & 0 & 0 & 0 & 0 & 0 & 1 & 0 & 0 & 0 & 0 & 0 & 0 & 0 & 0 \\
 0 & 0 & 0 & 0 & 0 & 0 & 0 & 0 & 0 & 0 & 1 & 0 & 0 & 0 & 0 & 0 & 0 & 0 \\
 0 & 0 & 0 & 0 & 0 & 0 & 0 & 0 & 0 & 0 & 0 & 1 & 0 & 0 & 0 & 0 & 0 & 0 \\
 0 & 0 & 0 & 0 & 0 & 0 & 0 & 0 & 0 & 0 & 0 & 0 & 1 & 0 & 0 & 0 & 0 & 0 \\
 0 & 0 & 0 & 0 & 0 & 0 & 0 & 0 & 0 & 0 & 0 & 0 & 0 & 1 & 0 & 0 & 0 & 0 \\
 0 & 0 & 0 & 0 & 0 & 0 & 0 & 0 & 0 & 0 & 0 & 0 & 0 & 0 & 1 & 0 & 0 & 0 \\
 -1 & 0 & 0 & 0 & 0 & 0 & 0 & 0 & 0 & 0 & 0 & 0 & 0 & 0 & 0 & 1 & 0 & 0 \\
 0 & 0 & 0 & 0 & 0 & 0 & 0 & 0 & 0 & 0 & 0 & 0 & 0 & 0 & 0 & 0 & 1 & 0 \\
 0 & 0 & 0 & 0 & 0 & 0 & 0 & 0 & 0 & 0 & 0 & 0 & 0 & 0 & 0 & 0 & 0 & 1 \\
 0 & -1 & 0 & 0 & 0 & 0 & 0 & 0 & 0 & 0 & 0 & 0 & 0 & 0 & 0 & 0 & 0 & 0 \\
 0 & 0 & -1 & 0 & 0 & 0 & 0 & 0 & 0 & 0 & 0 & 0 & 0 & 0 & 0 & 0 & 0 & 0 \\
 0 & 0 & 0 & 0 & 0 & 0 & 0 & 0 & 0 & 0 & 0 & 0 & 0 & 0 & 0 & 0 & 0 & 0 \\
 0 & 0 & 0 & -1 & 0 & 0 & 0 & 0 & 0 & 0 & 0 & 0 & 0 & 0 & 0 & 0 & 0 & 0 \\
 0 & 0 & 0 & 0 & -1 & 0 & 0 & 0 & 0 & 0 & 0 & 0 & 0 & 0 & 0 & 0 & 0 & 0 \\
 0 & 0 & 0 & 0 & 0 & -1 & 0 & 0 & 0 & 0 & 0 & 0 & 0 & 0 & 0 & 0 & 0 & 0 \\
 0 & 0 & 0 & 0 & 0 & 0 & -1 & 0 & 0 & 0 & 0 & 0 & 0 & 0 & 0 & 0 & 0 & 0 \\
 0 & 0 & 0 & 0 & 0 & 0 & 0 & -1 & 0 & 0 & 0 & 0 & 0 & 0 & 0 & 0 & 0 & 0 \\
 0 & 0 & 0 & 0 & 0 & 0 & 0 & 0 & -1 & 0 & 0 & 0 & 0 & 0 & 0 & 0 & 0 & 0 \\
\hline
 0 & 1 & 0 & 0 & 0 & 0 & 0 & 0 & 0 & 0 & 0 & 0 & 0 & 0 & 0 & 0 & 0 & 0 \\
 0 & 0 & 0 & 2 & 0 & 0 & 0 & 0 & 0 & 0 & 0 & 0 & 0 & 0 & 0 & 0 & 0 & 0 \\
 0 & 0 & 0 & 0 & 1 & 0 & 0 & 0 & 0 & 0 & 0 & 0 & 0 & 0 & 0 & 0 & 0 & 0 \\
 0 & 0 & 0 & 0 & 0 & 0 & 3 & 0 & 0 & 0 & 0 & 0 & 0 & 0 & 0 & 0 & 0 & 0 \\
 0 & 0 & 0 & 0 & 0 & 0 & 0 & 2 & 0 & 0 & 0 & 0 & 0 & 0 & 0 & 0 & 0 & 0 \\
 0 & 0 & 0 & 0 & 0 & 0 & 0 & 0 & 1 & 0 & 0 & 0 & 0 & 0 & 0 & 0 & 0 & 0 \\
 0 & 0 & 0 & 0 & 0 & 0 & 0 & 0 & 0 & 0 & 0 & 0 & 0 & 0 & 0 & 0 & 0 & 0 \\
 0 & 0 & 0 & 0 & 1 & 0 & 0 & 0 & 0 & 0 & 0 & 0 & 0 & 0 & 0 & 0 & 0 & 0 \\
 0 & 0 & 0 & 0 & 0 & 2 & 0 & 0 & 0 & 0 & 0 & 0 & 0 & 0 & 0 & 0 & 0 & 0 \\
 0 & 0 & 1 & 0 & 0 & 0 & 0 & 1 & 0 & 0 & 0 & 0 & 0 & 0 & 0 & 0 & 0 & 0 \\
 0 & 0 & 0 & 0 & 0 & 0 & 0 & 0 & 1 & 0 & 0 & 0 & 0 & 0 & 0 & 0 & 0 & 0 \\
 0 & 0 & 0 & 0 & 0 & 0 & 0 & 0 & 0 & 0 & 0 & 0 & 0 & 0 & 0 & 0 & 0 & 0 \\
 0 & 0 & 0 & 0 & 0 & 0 & 0 & 0 & 0 & -3 & 0 & 0 & 0 & 0 & 0 & 0 & 0 & 0 \\
 0 & 0 & 0 & 0 & 0 & 0 & 0 & 0 & 0 & 0 & -2 & 0 & 0 & 0 & 0 & 0 & 0 & 0 \\
 0 & 0 & 0 & 0 & 0 & 0 & 0 & 0 & 0 & 0 & 0 & -1 & 0 & 0 & 0 & 0 & 0 & 0 \\
 0 & 0 & 0 & 0 & 0 & 0 & 0 & 0 & 0 & 0 & 0 & 0 & -1 & 0 & 0 & 0 & 0 & 0 \\
 0 & 0 & 0 & 0 & 0 & 0 & 0 & 0 & 0 & 0 & 0 & 0 & 0 & -2 & 0 & 0 & 0 & 0 \\
 0 & 0 & 0 & 0 & 0 & 0 & 0 & 0 & 0 & 0 & 0 & 0 & 0 & 0 & -2 & 0 & 0 & 0 \\
 0 & 0 & 0 & 0 & 0 & 0 & 0 & 0 & 0 & 0 & 0 & 1 & 0 & 0 & 0 & 0 & -1 & 0 \\
 0 & 0 & 0 & 0 & 0 & 0 & 0 & 0 & 0 & 0 & 0 & 0 & 0 & 0 & 0 & 0 & 0 & 0 \\
 0 & 0 & 0 & 0 & 0 & 0 & 0 & 0 & 0 & 0 & 0 & 0 & 0 & 0 & 0 & 0 & 0 & 0 \\
 0 & 0 & 0 & 0 & 0 & 0 & 0 & 0 & 0 & 0 & 0 & 0 & 0 & 1 & 0 & 0 & 0 & 0 \\
 0 & 0 & 0 & 0 & 0 & 0 & 0 & 0 & 0 & 0 & 0 & 0 & 0 & 0 & 1 & 0 & 0 & 0 \\
 0 & 0 & 0 & 0 & 0 & 0 & 0 & 0 & 0 & 0 & 0 & 0 & 0 & 0 & 0 & 0 & 0 & 2
\end{array}
\right).
}
$$
The matrix of $\ud^{\tot,1}$ is the transpose of
$$
{\tiny
\left(
\begin{array}{cccccccccccccccccccccccc}
 -3 & 0 & 0 & 0 & 0 & 0 & 0 & 0 & 0 & 0 & 0 & 0 & 0 & 0 & 0 & 0 & 0 & 0 & 0 & 0 & 0 & 0 & 0 & 0 \\
 0 & -2 & 0 & 0 & 0 & 0 & 0 & 0 & 0 & 0 & 0 & 0 & 0 & 0 & 0 & 0 & 0 & 0 & 0 & 0 & 0 & 0 & 0 & 0 \\
 0 & 0 & -1 & 0 & 0 & 0 & 0 & 0 & 0 & 0 & 0 & 0 & 0 & 0 & 0 & 0 & 0 & 0 & 1 & 0 & 0 & 0 & 0 & 0 \\
 0 & 0 & 0 & -1 & 0 & 0 & 0 & 0 & 0 & 0 & 0 & 0 & 0 & 0 & 0 & 0 & 0 & 0 & 0 & 0 & 0 & 0 & 0 & 0 \\
 0 & 0 & 0 & 0 & -2 & 0 & 0 & 0 & 0 & 0 & 0 & 0 & 0 & 0 & 0 & 0 & 0 & 0 & 0 & 0 & 0 & 1 & 0 & 0 \\
 0 & 0 & 0 & 0 & 0 & -2 & 0 & 0 & 0 & 0 & 0 & 0 & 0 & 0 & 0 & 0 & 0 & 0 & 0 & 0 & 0 & 0 & 1 & 0 \\
 0 & 0 & 0 & 0 & 0 & 0 & 0 & 0 & 0 & 0 & 0 & 0 & 0 & 0 & 0 & 0 & 0 & 0 & 0 & 0 & 0 & 0 & 0 & 0 \\
 0 & 0 & 0 & 0 & 0 & 0 & 0 & 0 & 0 & 0 & 0 & 0 & 0 & 0 & 0 & 0 & 0 & 0 & -1 & 0 & 0 & 0 & 0 & 0 \\
 0 & 0 & 0 & 0 & 0 & 0 & 0 & 0 & -2 & 0 & 0 & 0 & 0 & 0 & 0 & 0 & 0 & 0 & 0 & 0 & 0 & 0 & 0 & 2 \\
 0 & 0 & 0 & 0 & 0 & 0 & 0 & 0 & 0 & 1 & 0 & 0 & 0 & 0 & 0 & 0 & 0 & 0 & 0 & 0 & 0 & 0 & 0 & 0 \\
 0 & 0 & 0 & 0 & 0 & 0 & 0 & -1 & 0 & 0 & 0 & 0 & 0 & 0 & 0 & 0 & 1 & 0 & 1 & 0 & 0 & 0 & 0 & 0 \\
 0 & 0 & 0 & 0 & 0 & 0 & 0 & 0 & 0 & 0 & 0 & 1 & 0 & 0 & 0 & 0 & 0 & 0 & 0 & 0 & -2 & 0 & 0 & 0 \\
 0 & 0 & 0 & 0 & 0 & 0 & 0 & 0 & 0 & 0 & 0 & 0 & 2 & 0 & 0 & 0 & 0 & 0 & 0 & 0 & 0 & 0 & 0 & 0 \\
 0 & 0 & 0 & 0 & 0 & 0 & 0 & 0 & 0 & 0 & 0 & 0 & 0 & 2 & 0 & 0 & 0 & 0 & 0 & 0 & 0 & -1 & 0 & 0 \\
 0 & 0 & 0 & 0 & 0 & 0 & 0 & 0 & 0 & 0 & 0 & 0 & 0 & 0 & 2 & 0 & 0 & 0 & 0 & 0 & 0 & 0 & -2 & 0 \\
 0 & 0 & 0 & 0 & 0 & 0 & 0 & 0 & 0 & 0 & 0 & 0 & 0 & 0 & 0 & 3 & 0 & 0 & 0 & 0 & 0 & 0 & 0 & 0 \\
 0 & 0 & 0 & 0 & 0 & 0 & 0 & -1 & 0 & 0 & 0 & 0 & 0 & 0 & 0 & 0 & 3 & 0 & 1 & 0 & 0 & 0 & 0 & 0 \\
 0 & 0 & 0 & 0 & 0 & 0 & 0 & 0 & 0 & 0 & 0 & 0 & 0 & 0 & 0 & 0 & 0 & 2 & 0 & 0 & 0 & 0 & 0 & -1 \\
\hline
 0 & 0 & 0 & 0 & 0 & 0 & 0 & 0 & 0 & 1 & 0 & 0 & 0 & 0 & 0 & 0 & 0 & 0 & 0 & 0 & 0 & 0 & 0 & 0 \\
 0 & 0 & 0 & 0 & 0 & 0 & 0 & 0 & 0 & 0 & 0 & 0 & 1 & 0 & 0 & 0 & 0 & 0 & 0 & 0 & 0 & 0 & 0 & 0 \\
 0 & 0 & 0 & 0 & 0 & 0 & 0 & 0 & 0 & 0 & 0 & 0 & 0 & 1 & 0 & 0 & 0 & 0 & 0 & 0 & 0 & 0 & 0 & 0 \\
 0 & 0 & 0 & 0 & 0 & 0 & 0 & 0 & 0 & 0 & 0 & 0 & 0 & 0 & 0 & 1 & 0 & 0 & 0 & 0 & 0 & 0 & 0 & 0 \\
 0 & 0 & 0 & 0 & 0 & 0 & 0 & 0 & 0 & 0 & 0 & 0 & 0 & 0 & 0 & 0 & 1 & 0 & 0 & 0 & 0 & 0 & 0 & 0 \\
 0 & 0 & 0 & 0 & 0 & 0 & 0 & 0 & 0 & 0 & 0 & 0 & 0 & 0 & 0 & 0 & 0 & 1 & 0 & 0 & 0 & 0 & 0 & 0 \\
 0 & 0 & 0 & 0 & 0 & 0 & 0 & 0 & 0 & 0 & 1 & 0 & 0 & 0 & 0 & 0 & 0 & 0 & 0 & -1 & 0 & 0 & 0 & 0 \\
 0 & 0 & 0 & 0 & 0 & 0 & 0 & 0 & 0 & 0 & 0 & 0 & 0 & 1 & 0 & 0 & 0 & 0 & 0 & 0 & 0 & -1 & 0 & 0 \\
 0 & 0 & 0 & 0 & 0 & 0 & 0 & 0 & 0 & 0 & 0 & 0 & 0 & 0 & 1 & 0 & 0 & 0 & 0 & 0 & 0 & 0 & -1 & 0 \\
 0 & 0 & 0 & 0 & 0 & 0 & 0 & -1 & 0 & 0 & 0 & 0 & 0 & 0 & 0 & 0 & 1 & 0 & 1 & 0 & 0 & 0 & 0 & 0 \\
 0 & 0 & 0 & 0 & 0 & 0 & 0 & 0 & 0 & 0 & 0 & 0 & 0 & 0 & 0 & 0 & 0 & 1 & 0 & 0 & 0 & 0 & 0 & -1 \\
 0 & 0 & 0 & 0 & 0 & 0 & 0 & 0 & 0 & 0 & 0 & -1 & 0 & 0 & 0 & 0 & 0 & 0 & 0 & 0 & 1 & 0 & 0 & 0 \\
 -1 & 0 & 0 & 0 & 0 & 0 & 0 & 0 & 0 & 0 & 0 & 0 & 0 & 0 & 0 & 0 & 0 & 0 & 0 & 0 & 0 & 0 & 0 & 0 \\
 0 & -1 & 0 & 0 & 0 & 0 & 0 & 0 & 0 & 0 & 0 & 0 & 0 & 0 & 0 & 0 & 0 & 0 & 0 & 0 & 0 & 0 & 0 & 0 \\
 0 & 0 & -1 & 0 & 0 & 0 & 0 & 0 & 0 & 0 & 0 & 0 & 0 & 0 & 0 & 0 & 0 & 0 & 0 & 0 & 0 & 0 & 0 & 0 \\
 0 & 0 & 0 & -1 & 0 & 0 & 0 & 0 & 0 & 0 & 0 & 0 & 0 & 0 & 0 & 0 & 0 & 0 & 0 & 0 & 0 & 0 & 0 & 0 \\
 0 & 0 & 0 & 0 & -1 & 0 & 0 & 0 & 0 & 0 & 0 & 0 & 0 & 0 & 0 & 0 & 0 & 0 & 0 & 0 & 0 & 0 & 0 & 0 \\
 0 & 0 & 0 & 0 & 0 & -1 & 0 & 0 & 0 & 0 & 0 & 0 & 0 & 0 & 0 & 0 & 0 & 0 & 0 & 0 & 0 & 0 & 0 & 0 \\
 0 & 0 & 0 & 0 & 0 & 0 & 0 & 0 & 0 & 0 & 0 & 0 & 0 & 0 & 0 & 0 & 0 & 0 & -1 & 0 & 0 & 0 & 0 & 0 \\
 0 & 0 & 0 & 0 & 0 & 0 & 0 & 0 & 0 & 0 & 0 & 0 & 0 & 0 & 0 & 0 & 0 & 0 & 0 & -1 & 0 & 0 & 0 & 0 \\
 0 & 0 & 0 & 0 & 0 & 0 & 0 & 0 & 0 & 0 & 0 & 0 & 0 & 0 & 0 & 0 & 0 & 0 & 0 & 0 & -1 & 0 & 0 & 0 \\
 0 & 0 & 0 & 0 & 0 & 0 & 0 & 0 & 0 & 0 & 0 & 0 & 0 & 0 & 0 & 0 & 0 & 0 & 0 & 0 & 0 & -1 & 0 & 0 \\
 0 & 0 & 0 & 0 & 0 & 0 & 0 & 0 & 0 & 0 & 0 & 0 & 0 & 0 & 0 & 0 & 0 & 0 & 0 & 0 & 0 & 0 & -1 & 0 \\
 0 & 0 & 0 & 0 & 0 & 0 & 0 & 0 & 1 & 0 & 0 & 0 & 0 & 0 & 0 & 0 & 0 & 0 & 0 & 0 & 0 & 0 & 0 & -1
\end{array}
\right).
}
$$
Since the rank of $\ud^{\tot,1}$ is 23, $\Hdr^2(X)=H^2(T^\bullet)$ has the right
dimension, and since the seventh column of its matrix is zero, the corresponding
basis vector is not in the image, hence $\Hdr^2(X)$ is spanned by
$$
\frac{\Om_{01}}{X_0X_1}=\frac{\ud X_1}{X_1}-\frac{\ud X_0}{X_0}\in K^1(H_0\cup H_1).
$$
The nullity of $\ud^{\tot,1}$ is 19, and a basis for its kernel is given by the
columns of
$$
{\tiny
\left(
\begin{array}{ccccccccccccccccccc}
 0 & 0 & 0 & 0 & 0 & 0 & 0 & 0 & 0 & -1 & 0 & 0 & 0 & 0 & 0 & 0 & 0 & 0 & 0 \\
 0 & 0 & 0 & 0 & 0 & 0 & 0 & 0 & -1 & 0 & 0 & 0 & 0 & 0 & 0 & 0 & 0 & 0 & 0 \\
 0 & 0 & 0 & 0 & 0 & 0 & 0 & -1 & 0 & 0 & 0 & 0 & 0 & 0 & 0 & 0 & 0 & 0 & 0 \\
 0 & 0 & 0 & 0 & 0 & 0 & -1 & 0 & 0 & 0 & 0 & 0 & 0 & 0 & 0 & 0 & 0 & 0 & 0 \\
 0 & 0 & 0 & 0 & 0 & -1 & 0 & 0 & 0 & 0 & 0 & 0 & 0 & 0 & 0 & 0 & 0 & 0 & 0 \\
 0 & 1 & 0 & 0 & 0 & 0 & 0 & 0 & 0 & 0 & 0 & 0 & 0 & 0 & 0 & 0 & 0 & 0 & 0 \\
 0 & 0 & 0 & 0 & 0 & 0 & 0 & 0 & 0 & 0 & 0 & 0 & 0 & 0 & 0 & 0 & 0 & 0 & 1 \\
 0 & 0 & 0 & 0 & -1 & 0 & 0 & -1 & 0 & 0 & 0 & 0 & 0 & 0 & 0 & 0 & 0 & 0 & 0 \\
 1 & 0 & 0 & 0 & 0 & 0 & 0 & 0 & 0 & 0 & 0 & 0 & 0 & 0 & 0 & 0 & 0 & 0 & 0 \\
 0 & 0 & 0 & 0 & 0 & 0 & 0 & 0 & 0 & 0 & 0 & 0 & 0 & 0 & 0 & 0 & 0 & -1 & 0 \\
 0 & 0 & 0 & 0 & 0 & 0 & 0 & 0 & 0 & 0 & 0 & -1 & 0 & 0 & 1 & 0 & 0 & 0 & 0 \\
 0 & 0 & 0 & -1 & 0 & 0 & 0 & 0 & 0 & 0 & 0 & 0 & 0 & 0 & 0 & 0 & 0 & 0 & 0 \\
 0 & 0 & 0 & 0 & 0 & 0 & 0 & 0 & 0 & 0 & 0 & 0 & 0 & 0 & 0 & 0 & -1 & 0 & 0 \\
 0 & 0 & -1 & 0 & 0 & -1 & 0 & 0 & 0 & 0 & 0 & 0 & 0 & -1 & 0 & 0 & 0 & 0 & 0 \\
 0 & 0 & 0 & 0 & 0 & 0 & 0 & 0 & 0 & 0 & 0 & 0 & -1 & 0 & 0 & 0 & 0 & 0 & 0 \\
 0 & 0 & 0 & 0 & 0 & 0 & 0 & 0 & 0 & 0 & 0 & 0 & 0 & 0 & 0 & -1 & 0 & 0 & 0 \\
 0 & 0 & 0 & 0 & 0 & 0 & 0 & 0 & 0 & 0 & 0 & 0 & 0 & 0 & -1 & 0 & 0 & 0 & 0 \\
 0 & 0 & 0 & 0 & 0 & 0 & 0 & 0 & 0 & 0 & -1 & 0 & 0 & 0 & 0 & 0 & 0 & 0 & 0 \\
 0 & 0 & 0 & 0 & 0 & 0 & 0 & 0 & 0 & 0 & 0 & 0 & 0 & 0 & 0 & 0 & 0 & 1 & 0 \\
 0 & 0 & 0 & 0 & 0 & 0 & 0 & 0 & 0 & 0 & 0 & 0 & 0 & 0 & 0 & 0 & 2 & 0 & 0 \\
 0 & 0 & 2 & 0 & 0 & 2 & 0 & 0 & 0 & 0 & 0 & 0 & 0 & 1 & 0 & 0 & 0 & 0 & 0 \\
 0 & 0 & 0 & 0 & 0 & 0 & 0 & 0 & 0 & 0 & 0 & 0 & 0 & 0 & 0 & 3 & 0 & 0 & 0 \\
 0 & 0 & 0 & 0 & 0 & 0 & 0 & 0 & 0 & 0 & 0 & 0 & 0 & 0 & 2 & 0 & 0 & 0 & 0 \\
 0 & 0 & 0 & 0 & 0 & 0 & 0 & 0 & 0 & 0 & 1 & 0 & 0 & 0 & 0 & 0 & 0 & 0 & 0 \\
 0 & 0 & 0 & 0 & 0 & 0 & 0 & 0 & 0 & 0 & 0 & 0 & 0 & 0 & 0 & 0 & 0 & 0 & 0 \\
 0 & 0 & 0 & 0 & 0 & 0 & 0 & 0 & 0 & 0 & 0 & 0 & 0 & 1 & 0 & 0 & 0 & 0 & 0 \\
 0 & 0 & 0 & 0 & 0 & 0 & 0 & 0 & 0 & 0 & 0 & 0 & 2 & 0 & 0 & 0 & 0 & 0 & 0 \\
 0 & 0 & 0 & 0 & 0 & 0 & 0 & 0 & 0 & 0 & 0 & 1 & 0 & 0 & 0 & 0 & 0 & 0 & 0 \\
 0 & 0 & 0 & 0 & 0 & 0 & 0 & 0 & 0 & 0 & 1 & 0 & 0 & 0 & 0 & 0 & 0 & 0 & 0 \\
 0 & 0 & 0 & -1 & 0 & 0 & 0 & 0 & 0 & 0 & 0 & 0 & 0 & 0 & 0 & 0 & 0 & 0 & 0 \\
 0 & 0 & 0 & 0 & 0 & 0 & 0 & 0 & 0 & 3 & 0 & 0 & 0 & 0 & 0 & 0 & 0 & 0 & 0 \\
 0 & 0 & 0 & 0 & 0 & 0 & 0 & 0 & 2 & 0 & 0 & 0 & 0 & 0 & 0 & 0 & 0 & 0 & 0 \\
 0 & 0 & 0 & 0 & 0 & 0 & 0 & 1 & 0 & 0 & 0 & 0 & 0 & 0 & 0 & 0 & 0 & 0 & 0 \\
 0 & 0 & 0 & 0 & 0 & 0 & 1 & 0 & 0 & 0 & 0 & 0 & 0 & 0 & 0 & 0 & 0 & 0 & 0 \\
 0 & 0 & 0 & 0 & 0 & 2 & 0 & 0 & 0 & 0 & 0 & 0 & 0 & 0 & 0 & 0 & 0 & 0 & 0 \\
 0 & -2 & 0 & 0 & 0 & 0 & 0 & 0 & 0 & 0 & 0 & 0 & 0 & 0 & 0 & 0 & 0 & 0 & 0 \\
 0 & 0 & 0 & 0 & 1 & 0 & 0 & 0 & 0 & 0 & 0 & 0 & 0 & 0 & 0 & 0 & 0 & 0 & 0 \\
 0 & 0 & 0 & 0 & 0 & 0 & 0 & 0 & 0 & 0 & 0 & 0 & 0 & 0 & 0 & 0 & 0 & 0 & 0 \\
 0 & 0 & 0 & 1 & 0 & 0 & 0 & 0 & 0 & 0 & 0 & 0 & 0 & 0 & 0 & 0 & 0 & 0 & 0 \\
 0 & 0 & 1 & 0 & 0 & 0 & 0 & 0 & 0 & 0 & 0 & 0 & 0 & 0 & 0 & 0 & 0 & 0 & 0 \\
 0 & 1 & 0 & 0 & 0 & 0 & 0 & 0 & 0 & 0 & 0 & 0 & 0 & 0 & 0 & 0 & 0 & 0 & 0 \\
 2 & 0 & 0 & 0 & 0 & 0 & 0 & 0 & 0 & 0 & 0 & 0 & 0 & 0 & 0 & 0 & 0 & 0 & 0
\end{array}
\right).
}
$$
The rank of $\ud^{\tot,0}$ is 17, and by some rank computations one sees that
the third and fourth column of the above matrix are a basis of the first
cohomology $\Hdr^1(X)=H^1(T^\bullet)$, which correspond to
$$
\left(-\frac{X_1X_2}{X_0^2},2\frac{X_2\Om_{01}}{X_0^3},\frac{\Om_{02}}{X_1^2}\right),
\left(-\frac{X_2^2}{X_0X_1},-\frac{X_1X_2\Om_{12}}{X_0^4},\frac{X_0X_2\Om_{02}}{X_1^4}\right)
$$
$$
\in K^0(H_0\cup H_1)\oplus K^1(H_0)\oplus K^1(H_1).
$$
Finally, the nullity of $\ud^{\tot,0}$ is 1, and the third last column cancels
with the first, so $\Hdr^0(X)=H^0(T^\bullet)$ is generated by
$$
(1,1)\in K^0(H_0)\oplus K^0(H_1).
$$

Now let us see if we can find the Hodge decomposition
$$
\Hdr^1(X)\simeq H^1(\Oh_X)\oplus H^0(\Om_X)
$$
in our double complex $K^{\bullet,\bullet}$. For that purpose, note that $\d^0$ is
the submatrix of $\ud^{\tot,1}$ above the line, and $\d^1$ is the transpose of
the submatrix of $\ud^{\tot,1}$ below the line. Since the $12^{\rm{th}}$ row of
$\d^0$ is zero, $H^1(\Oh_X)$ is spanned by
$$
\frac{X_2^2}{X_0X_1}\in K^0(H_0\cup H_1).
$$
Furthermore, one finds that the kernel of $\d^1$
and hence $H^0(\Om_X)$ is generated by
$$
\left(\frac{X_2\Om_{01}-X_1\Om_{02}}{X_3^3},\frac{\Om_{02}}{X_1^2}\right)=
\left(\frac{-\Om_{12}}{X_0^2},\frac{\Om_{02}}{X_1^2}\right)\in K^1(H_0)\oplus
K^1(H_1).
$$
So indeed, we find generaters for the sheaf cohomologies of $\Oh_X$ and $\Om_X$
in our vertical subcomplexes $K^{p,\bullet}$. It is interesting to note that
$\left(0,\frac{-\Om_{12}}{X_0^2},\frac{\Om_{02}}{X_1^2}\right)$ can also be chosen
as one vector in a basis of $\Hdr^1(X)$, but $(\frac{X_2^2}{X_0X_1},0,0)$ cannot.
In fact, 
it is easy to see that any element of the form $(f,0,0)$
in the kernel of $\ud^{\tot,1}$ is in the image of $\ud^{\tot,0}$.
\end{example}

\section{Testing Smoothness}\label{se:testSmoothness}
In this section we describe how one can test in parallel polynomial time
whether a closed projective variety $X$ is smooth.

Crucial is Proposition~\ref{prop:generation} implying that if~$X$ is smooth,
then $X$ is scheme-theoretically cut out by hypersurfaces of degree
$\le D=\deg X$.
Using the linear system of equations \eqref{eq:prem} one can compute
a vector space basis $f_1,\ldots,f_N$ of $I_{\le D}$, where $I=I(X)$.
Let $U_i$ be the open subset $X\setminus\mZ(X_i)$ for $0\le i\le n$.
Then (cf.\ \S\ref{ss:basics}) the tangent space of $X$ at each $x\in U_i$ is
$$
T_x X=\mZ(d_x f_1^i,\ldots,d_x f_N^i)\subseteq\C^n\simeq\proj^n\setminus\mZ(X_i).
$$
Hence, if we assume $X$ to be $m$-equidimensional and denote with $L_x^i$
the linear subspace $\mZ(d_x f_1^i,\ldots,d_x f_N^i)$, we have
\begin{equation}\label{eq:equivSmooth}
X\ \textrm{smooth}\quad\iff\quad\bigwedge_i\;\forall\, x\in U_i\ \dim L_x^i=m,
\end{equation}
which is the Jacobian criterion.
Indeed, if $X$ is not smooth at $x\in U_i$, then $\dim L_x^i\ge\dim T_x X > m$, 
since in general $T_x X\subseteq L_x^i$.

Now our algorithm reads as follows.
\\

\noindent {\bf Algorithm Smoothness Test}\\

\noindent {\tt input} $X$ given by homogeneous polynomials of degree $\le d$.
\begin{enumerate}
\item Compute the equidimensional decomposition $X=Z_0\cup\cdots\cup Z_n$,
where~$Z_m$ is either empty or $m$-equidimensional.
\item {\tt if} $Z_n\ne\emptyset$\, {\tt then output} ``Yes''.
\item {\tt for} $0\le m<m'<n$ {\tt do if} $Z_m\cap Z_{m'}\ne\emptyset$\, 
{\tt then output} ``No''.
\item Set $D:=d^n$. 
\item Compute a basis $f_1,\ldots,f_N$ of $I_{\le D}$, where $I=I(X)$.
\item {\tt for} $0\le m<n$ {\tt do}
\item \quad {\tt for} $0\le i\le n$ {\tt do}
\item \quad\quad Compute the matrix
 $A:=\left(\partder{f_\nu^i}{X_\mu}\right)_{\mu,\nu}$,
where $f_\nu^i$ is the dehomogenization of $f_\nu$ with respect to $X_i$.
\item \quad\quad Compute the Mulmuley polynomial $p(Z)$ of $A$
(see~\S\ref{ss:linAlg}), which lies in
$\C[X_0,\ldots,\hat{X_i},\ldots,X_n,T,Z]$.
Write $p(Z)=p_0+p_1 Z+\cdots+p_K Z^K$,
and let $F_1,\ldots,F_L\in\C[X_0,\ldots,\hat{X_i},\ldots,X_n]$ be the
coefficients of all $T^k$ in~$p_0,\ldots,p_m$.
\item \quad\quad {\tt if} $Z_m\cap\mZ(F_1,\ldots,F_L)\cap\{X_i\ne 0\}\ne\emptyset$\,
{\tt then output} ``No''.
\item {\tt output} ``Yes''.
\end{enumerate}

\begin{proposition}\label{prop:smoothnessTest}
The algorithm Smoothness Test is correct and can be implemented in parallel
time $(n\log d)^{\Oh(1)}$ and sequential time $d^{\Oh(n^4)}$.
\end{proposition}
\begin{proof}
Correctness:
If $X$ is smooth, then it clearly passes the test in step 3.
By~\eqref{eq:equivSmooth} we have $\dim L_x^i=m$ for all $m,i$ and all $x\in Z_m$.
Denote by $A_x$ the matrix $A$ evaluated at $x$, and similarly for $p(Z)$.
Then $L_x^i=\ker A_x$, hence
$$
\dim L_x^i>m\quad\iff\quad Z^{m+1}\,|\, p_x(Z)\quad\iff\quad F_1(x)=\cdots=F_L(x)=0.
$$
If $X$ is not smooth, then it doesn't pass the test in step 3 or some $Z_m$ is
not smooth. In the latter case at some point $x\in Z_m\cap U_i$ we will have
$\dim L_x^i>m$.

Analysis: All the algorithms we use are well-parallelizable. We therefore state
only the sequential time bounds.
The equidimensional decomposition in step 1 can be done in time $d^{\Oh(n^2)}$
with the algorithm of~\cite{giu:91}. For each~$m$, this algorithm returns
$d^{\Oh(n^2)}$ polynomials of degree bounded by $\deg Z_m=\Oh(d^n)$ whose
zero set is $Z_m$. Testing feasibility of a system of $r$ homogeneous equations
of degree $\overline{d}$ can be done in time $r(n\overline{d})^{\Oh(n)}$ using
the effective homogeneous Nullstellensatz.
Hence step 3 takes time $d^{\Oh(n^2)}$.
Sz{\'a}nt{\'o}'s algorithm in step 5 runs in time $d^{\Oh(n^4)}$, and clearly
$N=\Oh(D^{n+1})=d^{\Oh(n^2)}$.
Furthermore, the computation of the Mulmuley polynomial in step 9 can be done
in time $d^{\Oh(n^3)}$ by~\S\ref{ss:linAlg}, and we have $L=\Oh(N^2m)=d^{\Oh(n^2)}$
and $\deg F_i\le ND=d^{\Oh(n^2)}$. Thus step 10 takes time  $d^{\Oh(n^4)}$
by the affine effective Nullstellensatz.
\end{proof}

\section{Finding Generic Hyperplanes}
The algorithm for computing the cohomology of $X$ described 
in~\S\ref{ss:compCohHypCompl} and~\S\ref{se:patching} depends on a choice
of sufficiently generic
hyperplane sections $H_\nu$ of~$X$. More precisely, it is
required that $V:=H_0\cup\cdots\cup H_m$ is a hypersurface with normal
crossings in $X$, where $m=\dim X$. Note that as a consequence
$H_{i_0}\cup\cdots\cup H_{i_q}$ has normal crossings for each
tuple $i_0<\cdots <i_q$.
Here we describe how to find sufficiently generic hyperplanes
deterministically in parallel polynomial time.

Throughout this section we assume $X$ to be smooth, and let us first
assume that $X$ is $m$-equidimensional.
We will formulate a sufficient condition for normal crossings in terms of
transversality.
Recall that a linear subspace $L\subseteq\proj^n$ is called
{\em transversal} to $X$ in $x\in X\cap L$, written
$X\pitchfork_x L$, iff
$\dim(T_x X\cap T_x L)=\dim T_x X+\dim T_x L-n$.
Now let the hypersurfaces $H_0,\ldots,H_m$ be given by the linear
forms $\ell_0,\ldots,\ell_m\in\C[X_0,\ldots,X_n]$. Denote
$L_{i_0\cdots i_q}:=\mZ(\ell_{i_0},\ldots,\ell_{i_q})$
for all $0\le q\le m$ and all $0\le i_0<\cdots<i_q\le m$.
\begin{lemma}\label{lemma:normCrossingCond}
If $\ell_0,\ldots,\ell_m$ are linearly independent and for all
$0\le q\le m$ and all $0\le i_0<\cdots<i_q\le m$ we have 
\begin{equation}\label{eq:normalCrossingCond}
\forall x\in X\cap L_{i_0\cdots i_q}\colon\ X\pitchfork_x L_{i_0\cdots i_q},
\end{equation}
then~$V\subseteq X$ is a hypersurface with normal crossings.
\end{lemma}
\begin{proof}
Suppose that the condition~\eqref{eq:normalCrossingCond} holds. First note
by choosing $q=0$ that $L_i=\mZ(\ell_i)$ is transversal to $X$
at all~$x$, thus $V$ is indeed a hypersurface. Furthermore, $H_i=X\cap L_i$ is smooth in
$x$, so that~$x$ lies in only one irreducible component of $H_i$, and
$\ell_i\in\Oh_{X,x}$ is a local equation of that component. 
By transversality we have $\dim(T_x X\cap T_x L_{i_0\cdots i_q})=m-q-1$.
But $T_x X\cap T_x L_{i_0\cdots i_q}$ is the kernel of the linear map
$\var:=(d_x\ell_{i_0},\ldots, d_x\ell_{i_q})\colon T_x X\rightarrow\C^{q+1}$,
which thus must be surjective. Hence $d_x\ell_{i_0},\ldots, d_x\ell_{i_q}$
are linearly independent on $T_x X$, which proves the claim.
\end{proof}

In order to work with condition~\eqref{eq:normalCrossingCond} algorithmically,
we introduce some notation.
Set $I:=I(X)$ and $D:=\deg X$. Recall from~\S\ref{se:testSmoothness} that if 
$f_1,\ldots,f_N$ is a vector space basis of~$I_{\le D}$, then
$$
T_x X=\mZ(d_x f_1^i,\ldots,d_x f_N^i)\subseteq\C^n
$$
for all $x\in U_i=X\setminus \mZ(X_i)$ and $0\le i\le n$.
For each tuple $i_0<\cdots <i_q$ and each~$i$ we define the
matrix
\begin{equation}\label{eq:matrix}
A^i_{i_0\cdots i_q}:=
\left(\begin{array}{c}
d_x f_1^i \\
\vdots \\
d_x f_N^i \\
d_x\ell^i_{i_0} \\
\vdots \\
d_x\ell^i_{i_q} \\
\end{array}\right)\in\C[X_0,\ldots,\hat{X_i},\ldots,X_n]^{(N+q+1)\times n}.
\end{equation}
Then the kernel of $A^i_{i_0\cdots i_q}$ is the kernel of $\var$ of the proof
of Lemma~\ref{lemma:normCrossingCond}. 
Assume that $\ell_0,\ldots,\ell_m$ are linearly independent. Then
condition~\eqref{eq:normalCrossingCond} is eqivalent to the statement that
the nullity of $A^i_{i_0\cdots i_q}$ is $m-q-1$, its minimal possible value,
at each point $x\in U_i\cap L_{i_0\cdots i_q}$. 
Note that this condition also implies the linear independence.
Now let $p(Z)$ be the Mulmuley polynomial of $A^i_{i_0\cdots i_q}$, which lies
in $\C[X_0,\ldots,\hat{X_i},\ldots,X_n,T,Z]$.
Let $F_1,\ldots,F_L\in\C[X_0,\ldots,\hat{X_i},\ldots,X_n]$ be the
coefficients of all $T^k$ in the coefficient of $Z^{m-q}$ in $p(Z)$.
Then a sufficient condition for~\eqref{eq:normalCrossingCond} is
\begin{equation}\label{eq:normalCrossingFirstOrderCond}
\bigwedge_i U_i\cap L_{i_0\cdots i_q}\cap\mZ(F_1,\ldots,F_L)\ne\emptyset.
\end{equation}

Using this formula we can prove
\begin{proposition}\label{prop:chooseHyperplanes}
Given polynomials of degree $\le d$ defining a smooth subvariety
$X\subseteq\proj^n$ of dimension $m$, one can compute in
parallel time $(n\log d)^{\Oh(1)}$ and sequential time $d^{{\Oh(n^4)}}$ linear
forms
$\ell_0,\ldots,\ell_m$ such that $V=\bigcup_j H_j$ is a hypersurface with
normal crossings, where $H_j=\mZ_X(\ell_j)$.
\end{proposition}
\begin{proof}
First we set $D:=d^n$ and compute a basis $f_1,\ldots,f_N$ of $I_{D}$. Then
we compute the equidimensional components $Z_m$ of $X$.
We find the linear forms $\ell_0,\ldots,\ell_m$ successively, one at a time.
So assume that $\ell_0,\ldots,\ell_{j-1}$ have been already found, and take
$\ell_j=\a_0X_0+\cdots +\a_n X_n$ with indeterminate coefficients
$\a=(\a_0,\ldots,\a_n)$.
Now consider the conjunction of the
conditions~\eqref{eq:normalCrossingFirstOrderCond} for all $m\le\dim X$ and
$i_0<\cdots<i_q=j$, which is a first order formula with free variables~$\a$.
Note that here one has to take $U_i=Z_m\cap\{X_i\ne 0\}$. 
By quantifier elimination compute an equivalent quantifier-free formula
$\Phi(\a)$ in disjunctive normal form. Let $G_1,\ldots,G_M$ be all polynomials accuring $\Phi(\a)$. Since $\Phi(\a)$ is satisfied for generic $\a$, it is easy
to see that $G_{\nu}(\a)\ne 0$ for all~$\nu$
implies~$\Phi(\a)$~\cite[proof of Theorem 3.8]{bus:10}.
Let $\d$ be the maximal degree of the~$G_\nu$.
Now take a set $S\subseteq\C$ of cardinality $>Mn\d$ and test for all
$b\in S$ in parallel, whether $P_\nu(b):=G_{\nu}(1,b,\ldots,b^n)\ne 0$ for all $1\le\nu\le M$.
Since $P_\nu$ is a univariate polynomial of degree $\le n\d$, there must exist
a successful $b$. Then we can take $\ell_j=X_0+bX_1+\ldots+b^nX_n$.

Analysis: The computation of $f_1,\ldots,f_N$, of the equidimensional
decomposition, and of the Mulmuley polynomial can be done within the claimed
time bounds. 
Recall that $N,L$, the degrees of the defining equations for $Z_m$, as well as
$\deg F_i$ are of order $d^{\Oh(n^2)}$.
Condition~\eqref{eq:normalCrossingFirstOrderCond} is a universal
first order formula with $\Oh(n)$ free and bounded variables,
and $d^{\Oh(n^2)}$ atomic formulas involving polynomials of degree
$d^{\Oh(n^2)}$. According to~\cite{fgm:90}, one can eliminate the
universal quantifier and hence compute the polynomials $G_\nu$ in parallel
time $(n\log d)^{\Oh(1)}$ and sequential time $d^{{\Oh(n^4)}}$. Furthermore,
$M$ and $\d$ are also bounded by $d^{\Oh(n^4)}$. Hence the cardinality of
the set $S$ is $d^{\Oh(n^4)}$ and our claim follows.
\end{proof}

Theorem \ref{thm:main} follows from the Propositions \ref{prop:smoothnessTest},
\ref{prop:chooseHyperplanes}, and \ref{prop:cohomology}.


\end{document}